\documentclass[a4paper, 12pt]{amsart}
\usepackage[margin=2.65cm]{geometry}
\usepackage{tikz, tikz-cd}
\tikzcdset{arrow style=tikz,
           diagrams={>=Straight Barb}
           }
\usetikzlibrary{arrows.meta}
\usepackage{amssymb,amsrefs,mathtools,tensor}

\newcommand{\id}{\operatorname{id}}
\newcommand{\PI}{\operatorname{PIso}}

\newcommand{\CC}{\mathbb{C}}

\newcommand{\NN}{\mathbb{N}}
\newcommand{\ZZ}{\mathbb{Z}}

\newcommand{\Cc}{\mathcal{C}}

\newcommand{\Jj}{\mathcal{J}}

\newcommand{\Nn}{\mathcal{N}}
\newcommand{\Oo}{\mathcal{O}}

\newcommand{\lZ}[1]{Z(#1]}
\newcommand{\rZ}[1]{Z[#1)}
\newcommand{\bZ}[1]{Z(#1)}

\newtheorem{thm}{Theorem}[section]
\newtheorem{cor}[thm]{Corollary}

\newtheorem{lem}[thm]{Lemma}
\newtheorem{prp}[thm]{Proposition}

\theoremstyle{definition}
\newtheorem{dfn}[thm]{Definition}

\theoremstyle{remark}
\newtheorem{rmk}[thm]{Remark}
\newtheorem{exm}[thm]{Example}

\numberwithin{equation}{section}

\makeatletter
\@namedef{subjclassname@2020}{%
  \textup{2020} Mathematics Subject Classification}
\makeatother

\tikzstyle{vertex}=[circle]
\tikzstyle{goto}=[->,shorten >=1pt,>=stealth,semithick]

\begin{document}

\title[Katsura-Exel-Pardo actions and their limit spaces]{Katsura-Exel-Pardo self-similar actions, Putnam's binary factors, and their limit spaces}

\author[Hume]{Jeremy B. Hume}
\author[Whittaker]{Michael F. Whittaker}

\address{Jeremy B. Hume and Michael F.
Whittaker \\ School of Mathematics and Statistics  \\
University of Glasgow}
\email{jeremybhume@gmail.com, Mike.Whittaker@glasgow.ac.uk}

\thanks{Hume was supported by the European Research Council (ERC) under the European Union's Horizon 2020 research and innovation programme (grant agreement No. 817597).\\ \indent We'd like to thank Enrique Pardo for comments on an early draft.}

\begin{abstract}
We show that the dynamical system associated by Putnam to a pair of graph embeddings is identical to the shift map on the limit space of a self-similar groupoid action on a graph. Moreover, performing a certain out-split on said graph gives rise to a Katsura-Exel-Pardo groupoid action on the out-split graph whose associated limit space dynamical system is conjugate to the previous one. We characterise the self-similar properties of these groupoids in terms of properties of their defining data, two matrices $A$, $B$. We prove a large class of the associated limit spaces are bundles of circles and points which fibre over a totally disconnected space, and the dynamics restricted to each circle is of the form $z\to z^{n}$. Moreover, we find a planar embedding of these spaces, thereby answering a question Putnam posed in his paper.
\end{abstract}



\maketitle

\centerline{\em Dedicated to our late colleague, mentor, and friend Iain Raeburn.}

\section{Introduction}\label{sec:intro}

One of the most beautiful aspects of self-similar group theory are its connections, discovered by Nekrashevych \cite{Nekrashevych:Self-similar}, to the theory of dynamical systems. To any contracting self-similar group, one can construct its \emph{limit dynamical system}, which is a self map of a compact metric space whose dynamical properties are governed by the properties of the self-similar group, and vice-versa. Many natural dynamical systems arise as examples, for instance hyperbolic post-critically finite rational maps acting on their Julia sets. This description was used, for instance, to solve the Twisted Rabbit Problem \cite{BN}.

We prove that Putnam's binary factors of subshifts of finite type \cite{Putnam:Binary} arise as limit dynamical systems of self-similar \emph{groupoids}.

An \emph{embedding pair} consists of two directed graphs $H$ and $E$, along with a pair of embeddings $\xi^0, \xi^1:H \hookrightarrow E$ satisfying certain conditions. Putnam's factor is obtained by identifying two (one-sided) infinite paths in $E$ that arise from the same path in $H$ embedded along two binary sequences of embeddings that are related through carry over in binary addition. He then proves that the natural extension of these dynamical systems are Smale spaces, computes their homology, as well as the $K$-theory of the associated $C^{*}$-algebras. As a corollary, Putnam proves that these $C^{*}$-algebras exhaust all possible Ruelle algebras arising from irreducible Smale spaces \cite[Theorem~6.5]{Putnam:Binary}. 

We show that Putnam's construction naturally defines a self-similar groupoid action on a graph. Moreover, we prove, in Theorem \ref{thm:equality}, that the limit dynamical system of the self-similar groupoid action is identical (not just conjugate) to Putnam's dynamical system. Through our approach to studying these systems, we are able to remove one of Putnam's standing hypotheses and weaken his requirement of the graph $E$ being primitive to having no sources, see Section \ref{Putnam's construction}. 

An interesting corollary of our construction that follows from \cite[Theorem~6.5]{Putnam:Binary} and \cite[Corollary~8.5]{BBGHSW} is that the class of (stabilised) $C^{*}$-algebras associated to contracting and regular self-similar groupoids acting on strongly connected finite graphs is equal to the class of Ruelle algebras associated to irreducible Smale spaces.

This result should be compared with Katsura's seminal paper \cite{Katsura:class_II}, where he proved that all Kirchberg algebras can be realised, up to strong Morita equivalence, as certain $C^*$-algebras associated with two integer matrices $A$, $B$. While studying these algebras, Exel and Pardo \cite{Exel-Pardo:Self-similar} realised that they arise from self-similar group actions, with the finite alphabet replaced by a (possibly infinite) graph. 

Using Kitchens' out-split construction for graphs \cite{Kitchens}, which we extend to self-similar groupoid actions, we prove that every self-similar groupoid action coming from an embedding pair can be out-split to a Katsura-Exel-Pardo action. Therefore, Putnam's dynamical systems are topologically conjugate to the limit dynamical system of a Katsura-Exel-Pardo action. This explains the similarity in the $K$-theory results \cite[Theorem 6.1 and 6.2]{Putnam:Binary} and Theorem \ref{Kirchberg} due to Katsura and Exel-Pardo.

The matrices $A$, $B$ arising from this out-split satisfy some relations between them that ensure the associated self-similar groupoid is contracting and regular. These conditions are dynamically important, as contracting guarantees the limit space is Hausdorff and, assuming contracting, regular is equivalent to the limit dynamical system being an expanding local homeomorphism (see \cite[Proposition~4.8]{BBGHSW}). Moreover, the $KK$-duality results of \cite{BBGHSW} may be applied in this setting. We characterise these properties in terms of properties of the matrices $A$ and $B$.
 
Given a self-similar groupoid action on a graph, we show the connected component space of the limit space can be identified with the quotient of the infinite path space of the graph by a natural equivalence relation. We use this description to prove, for a large class of Katsura-Exel-Pardo actions, the connected components in their limit spaces are circles and points, analogous to Putnam's result \cite[Corollary~7.7]{Putnam:Binary}. We identify the dynamics on the circle components as $z\to z^{n}$, where $n\in\mathbb{N}$ can vary, dependent on the value of $z$ under a natural factor map to a subshift of finite type. Thus, ``Katsura-Exel-Pardo systems'' exhibit interesting interplay of $0$ and $1$ dimensional dynamics.

This description provided the impetus to look for a planar embedding of the limit space for such Katsura-Exel-Pardo actions, to better understand how the circles and points are configured. The embedding is reminiscent of a solar system trajectory, with planets orbiting a star and moons orbiting the planet, but ad infinitum. As a corollary, we prove that Putnam's dynamical systems embed into the plane, answering Putnam's question (see \cite[Question 7.10]{Putnam:Binary}).

The paper is organised as follows. In Section \ref{Self-similar groupoid actions on graphs and unital Katsura algebras} we provide background on self-similar groupoid actions and their limit dynamical systems. Section \ref{Katsura groupoids} introduces Katsura's construction and Exel and Pardo's realisation of these as Katsura-Exel-Pardo groupoid actions on graphs, and contains our matrix characterisation of when they are contracting and regular. Section \ref{Putnam's construction} introduces Putnam's binary factors of subshifts of finite type, and we prove that they are limit dynamical systems of certain self-similar groupoid actions on graphs. Section \ref{out-splits} defines out-splits and uses them, along with the previous result, to show that Putnam's dynamical systems are topologically conjugate to the limit space dynamical systems of Katsura-Exel-Pardo actions over certain out-split graphs. Section \ref{Bundles of odometers} contains our results on the connected components of limit spaces. The final section proves that regular Katsura-Exel-Pardo systems with $B \in M_N(\{0,1\})$ embed into the plane, and hence, so do Putnam's dynamical systems.

\section{Self-similar groupoid actions on graphs}\label{Self-similar groupoid actions on graphs and unital Katsura algebras}

In this section we describe self-similar groupoid actions on finite directed graphs and their properties. These generalise the notion of a self-similar group introduced by Bartholdi, Grigorchuk, Nekrashevych and others.

\subsection{Directed graphs and their path spaces} We quickly introduce directed graphs, for a detailed treatment see Raeburn's seminal book \cite{Raeburn:Graph_algebras}.

A \emph{directed graph} $E$ is a quadruple $E = (E^0, E^1, r, s)$ consisting of two sets $E^0$ and $E^1$ along with two functions $r,s : E^1 \to E^0$ called the range and source map, respectively. Elements in $E^0$ are \emph{vertices} and elements in $E^1$ are \emph{edges}. We think of an edge $e$ as a directed arrow from its source vertex $s(e)$ to its range $r(e)$.

Perhaps the most important aspect of a directed graph is its path space. A finite \emph{path} $\mu$ in a directed graph $E$ is either a vertex $\mu = v$, or a finite sequence of edges $\mu = e_1 \dots e_n$ such that $s(e_i) = r(e_{i+1})$ for all $i\leq n-1$. Let $E^n = \{e_1 \dots e_n \mid e_i \in E^1, s(e_i) = r(e_{i+1})\}$ denote the paths of length $n$ in $E$. We then let $E^* \coloneqq \bigcup^\infty_{n=0} E^n$ denote the set of all finite paths in $E$. For a path $\mu=\mu_1 \ldots \mu_n$ in $E^n$, let $r(\mu) =r(\mu_1)$ and $s(\mu)=s(\mu_n)$. For $\mu \in E^*$ and $X \subseteq E^*$, we define 
\[ \mu X = \{\mu \nu \mid \nu
\in X, s(\mu) = r(\nu)\} \quad \text{ and } \quad X \mu = \{\nu\mu \mid \nu \in X, r(\mu) =
s(\nu)\}.
\]
We then have $\mu X \nu = \mu X \cap X \nu$.

A graph is \emph{finite} if both $E^0$ and $E^1$ are finite. A graph is \emph{strongly connected} if, for all $v,w \in E^0$, the set $v E^* w$ is nonempty. Notice that if $E$ is strongly connected, then $vE^1$ and $E^1v$ are nonempty for all $v \in E^0$, unless $E$ is the graph with one vertex and no edges. We say a vertex $v$ in a graph $E$ is a \emph{source} if $vE^1=\emptyset$ and a \emph{sink} if $E^1v=\emptyset$.

In this paper, we will need to work with both left-, right-, and bi-infinite paths in a graph $E$. Thus, we define:
\begin{itemize}
\item $E^{+\infty} \coloneqq  \{e_1 e_2 e_3 \dots \mid e_i \in E^1,\text{ } s(e_i) = r(e_{i+1})\text{ for all }i\}$,
\item $E^{-\infty} \coloneqq  \{\dots e_{-3} e_{-2} e_{-1} \mid e_i \in E^1,\text{ } s(e_i) = r(e_{i+1})\text{ for all }i\}$, and
\item $E^\ZZ \coloneqq  \{\dots e_{-2} e_{-1} e_0 e_1 e_2 \dots \mid e_i \in E^1,\text{ } s(e_i) = r(e_{i+1})\text{ for all }i\}$.
\end{itemize}
As usual, we endow these spaces with the product topology, with a basis of cylinder sets. These are indexed by finite paths in each of the three spaces, so we will distinguish them as follows. Define for
\begin{itemize}
\item $E^{+\infty}$: when $\mu \in E^n$, let $\rZ{\mu} \coloneqq \{x \in E^{+\infty} \mid x_1\dots x_n = \mu\}$, for
\item $E^{-\infty}$:  when $\mu \in E^n$, let $\lZ{\mu} \coloneqq \{x \in E^{-\infty} \mid x_{-n} \dots x_{-1} = \mu\}$, and for
\item $E^\ZZ$: when $n \ge 0$ and $\mu  \in E^{2n+1}$, let $\bZ{\mu} \coloneqq \{x \in E^\ZZ \mid x_{-n} \dots x_n = \mu\}$.
\end{itemize}
If $x$ is an element in any of the spaces above and $m < n \in \ZZ$ appropriately chosen for the space in question, we define
\[
\mu[m,n] \coloneqq \mu_m \mu_{m+1} \ldots \mu_n.
\]

\subsection{Self-similar actions of groupoids on graphs}\label{subsec:ssa}

Katsura-Exel-Pardo actions are a type of self-similar groupoid action on a graph, so we take a few paragraphs to introduce them. For further details see \cite{Laca-Raeburn-Ramagge-Whittaker:Equilibrium} . 

Suppose $E$ is a directed graph. Given $v,w  \in E^0$, a \emph{partial isomorphism} of $E^*$ is a bijection $g : vE^* \to wE^*$ that is length and path preserving in the sense that $|g(\mu)|=|\mu|$ and $g(\mu e) \in
g(\mu)E^1$ for all $\mu  \in E^*$ and $e  \in E^1$ satisfying $s(\mu) = r(e)$. We shall use the notation $g\cdot\mu \coloneqq g(\mu)$ to reduce the number of parentheses. These two conditions are equivalent to the following property: $g$ is length preserving, and for every $\mu\in vE^{*}$, there is a partial isomorphism $h:s(\mu)E^{*}\to s(g
\cdot \mu)E^{*}$ such that 
\begin{equation}\label{eq:restrictionpiso}
    g \cdot(\mu\nu) = (g\cdot \mu)(h \cdot \nu)\text{ for all $\nu \in s(\mu)E^*$.}
\end{equation}

We write $h = g|_{\mu}$, as it is uniquely defined by the above property, and call it the \emph{restriction} of $g$ to $\mu$.

Let $\PI(E^*)$ denote the set of all partial isomorphisms of $E^*$, which is itself a groupoid with units $\id_v: vE^* \to vE^*$ defined by $\id_v(\mu)=\mu$ for all $\mu \in vE^*$ and multiplication given by composition of maps. Since units are associated with vertices, we go ahead and identify the unit space of $\PI(E^*)$ with $E^0$. The \emph{isotropy group} of a unit $v \in E^0$ is the set of partial isomorphisms from $vE^*$ to $vE^*$.

Given a partial isomorphism $g : vE^* \to wE^*$, define its \emph{domain} to be $d(g) = v$ and its \emph{codomain} to be $c(g) = w$. That is, we are renaming the range and source maps in the groupoid $\PI(E^*)$ by the terms codomain and domain since the symbols $s$ and $r$ are already in use. 

Restriction and multiplication of elements satisfy several relations, which we record in the following lemma. 

\begin{lem}[{\cite[Lemma~3.4 and Proposition~3.6]{Laca-Raeburn-Ramagge-Whittaker:Equilibrium}}]\label{properties SSA}
Let $E$ be a finite directed graph. For $(g,h) \in
\PI(E^*)^{(2)}$, $\mu \in d(g)E^*$, $\nu \in s(\mu)E^*$ and $\eta \in c(g)E^*$, we
have
\begin{enumerate}
    \item $r(g \cdot \mu) = c(g)$ and $s(g \cdot \mu) = g|_\mu \cdot
        s(\mu)$;
    \item $g|_{\mu\nu} = (g|_\mu)|_\nu$;
    \item $\id_{r(\mu)}|_\mu = \id_{s(\mu)}$;
    \item $(hg)|_\mu = (h|_{g \cdot \mu})(g|_\mu)$; and
    \item $g^{-1}|_{\eta} = (g|_{g^{-1} \cdot \eta})^{-1}$.
\end{enumerate}
\end{lem}

A groupoid $G$ with unit space $E^0$ \emph{acts on} $E^*$ if there is a groupoid homomorphism $\phi : G \to \PI(E^*)$ that restricts to the identity map on $E^0$. Define $\text{Ker}(\phi) = \phi^{-1}(E^{0})$, which is a normal sub-groupoid of $G$. We say $G$ acts \emph{faithfully} on $E^*$ if $\phi$ is injective, or in other words $\text{Ker}(\phi) = E^0$. We will write $g \cdot \mu$ in place of $\phi(g)(\mu)$.

\begin{dfn}\label{SSGA}
Suppose $E = (E^0, E^1, r, s)$ is a directed graph, and $G$ is a groupoid with
unit space $E^0$ and a faithful action $\phi:G\to \PI(E^*)$. We say $(G, E)$ is a \emph{self-similar groupoid action} if for every $g\in G$ and $e\in d(g)E^*$, there is $h\in G$ such that $\phi(g)|_{e} = \phi(h)$. We will write $g|_{e} \coloneqq h$ and call it the \emph{restriction of $g$ to $e$}. Using Lemma \ref{properties SSA}, $g|_{\mu}\in G$ for $\mu\in d(g)E^{n}$ is well defined and satisfies
\begin{equation}\label{eq:restriction}
    g \cdot(\mu\nu) = (g\cdot \mu)(g|_{\mu} \cdot \nu)\text{ for all $\nu \in s(\mu)E^*$}
\end{equation}
Moreover, all the conclusions in Lemma \ref{properties SSA} hold for the restriction and multiplication when $\PI(E^*)$ is replaced with $G$.
\end{dfn}

If the action of $G$ on $E^*$ is range preserving, then $G_v \coloneqq \{g \in G \mid d(g)=v\}$ is a group and hence $G$ is a \emph{group bundle}. In that case we say $(G,E)$ is a \emph{self-similar group bundle on $E$}. As we will see later, this holds for all Katsura-Exel-Pardo actions by definition.

It will be useful in this paper to work with non-faithful actions of groupoids by partial isomorphisms whose faithful quotient is self-similar. These should be considered as \emph{non-faithful self-similar groupoids}. However, such a term is an oxymoron, so we shall name them as below.

\begin{dfn}
\label{dfn:action_restriction}
Let $E$ be a directed graph and $G$ a groupoid with unit space $E^{0}$. An \emph{action-restriction pair} for $(G,E)$ is a map
    \begin{equation}
G \tensor[_{d}]{\times}{_{r}} E^1 \ni (g,e) \to (g\cdot e, g|_e)\in  E^1 \tensor[_{s}]{\times}{_{c}} G
\end{equation} 
such that:
\begin{itemize}
\item[(A0)] $r(g\cdot e) = c(g)$ and $d(g|_{e}) = s(e)$ for every $(g,e)\in G \tensor[_{d}]{\times}{_{r}} E^1$;
\item[(A1)] $r(e)\cdot e = e$ and $r(e)|_{e} = s(e)$ for every $e\in vE^{1}$;
\item[(A2)] $gh\cdot e = g
\cdot (h\cdot e)$ for every $(g,h)\in G^{(2)}$, and $e\in d(g)E^1$;
\item[(A3)] $g^{-1}|_{e} = (g|_{g^{-1}\cdot e})^{-1}$ for every $g\in G$ and $e\in c(g)E^1$.

\end{itemize}

We will usually denote an action-restriction pair by $(G,E)$ when there is no ambiguity.
\end{dfn}

If we replace $G$ in the definition above with a finite set $A$ with $E^{0}\subseteq A$ and retracts $c,d:A\to E^{0}$, then this is the notion of an automaton defined in \cite{Laca-Raeburn-Ramagge-Whittaker:Equilibrium}, and one checks that such a pair extends to an action-restriction pair of the free groupoid associated to $(A,c,d)$. Katsura-Exel-Pardo actions are, in general, not generated from an automaton but an action-restriction pair defined above.

If we consider a group $G$, a finite directed graph $E$, an action $\sigma:G\times E^{1}\to E^{1}$ and a \emph{one-cocycle} $\varphi: G\times E^{1}\to E^{1}$ satisfying Exel and Pardo's conditions in \cite[Section~2.3]{Exel-Pardo:Self-similar}, then such a pair defines an action-restriction pair on the group bundle $G\times E^{0} =\{g_{v}\mid g\in G, v\in E^{0}\}$ by $$(g_{v},e)\to (\sigma(g,e), (\varphi(g,e)_{s(e)}).$$

An action-restriction pair for $(G,E)$ defines a (not-necessarily faithful) action $\phi:G\to \PI(E^{*})$: for $g\in G$ and $\mu = e\nu\in d(g)E^{n}$, we inductively define
\begin{equation}\label{eq:actionrestriction}
    g \cdot \mu = (g\cdot e)(g|_{e} \cdot \nu).
\end{equation}
Note that if $g\in\text{Ker}(\phi)$ and $e\in d(g)E^{1}$, then for all $\nu\in d(g|_{e})E^{*} = s(e)E^{*}$, we have $e g|_{e}\cdot\nu = g\cdot (e\nu) = e\nu$ and hence $g|_{e}\cdot\nu = \nu$. Therefore, $g|_{e}\in \text{Ker}(\phi)$.

It follows that if $q:G\to G_{\phi} \coloneqq G/\text{ker}(\phi)$ is the quotient map, then $q(g|_{e}) = q(g)|_{e}$ for all $g\in G$ and $e\in d(g)E^{1} = d(q(g))E^{1}$. Hence, the induced action $G_{\phi}\to \PI(E^*)$ is self-similar.

The case that the self-similar action comes from an action-restriction pair induced by an Exel-Pardo action as above is covered in more detail in \cite[Appendix~A]{Laca-Raeburn-Ramagge-Whittaker:Equilibrium}.

Every self-similar groupoid is an example of an action-restriction pair. Moreover, self-similar groupoids are in one-one correspondence with action-restriction pairs whose induced action as partial isomorphisms is faithful. 

\subsection{Properties and limit spaces of self-similar groupoid actions on graphs}\label{sec: definitions}
In this section we recall standard properties and constructions associated to action-restriction pairs and self-similar group(oid)s acting on the path space of a graph.

\begin{dfn}\label{dfn:contracting} Let $G$ be a groupoid and $E$ a finite directed graph. An action-restriction pair $(G, E)$ is \emph{contracting} if there exists a finite subset $F \subseteq G$ so that, for every $g  \in G$, there is an $n \ge 0$ such that $g|_{\mu}\in F$ for every path $\mu\in E^{k}$, $k\geq n$. Such a subset $F$ is called a \emph{contracting core of $(G,E)$}. The \emph{nucleus} of $(G,E)$ is the set
\[
    \Nn\coloneqq \bigcap\{F \subseteq G \mid F\text{ is a contracting core for }(G, E)\}.
\]
\end{dfn}

\begin{dfn}{\cite[Definition 6.1]{Nekrashevych:Cstar_selfsimilar}}\label{def:regular}
Let $(G, E)$ be an action restriction pair. Then, $(G, E)$ is \emph{regular} if, for every $g  \in G$, there is $K\in\mathbb{N}$ such that if $g\cdot \mu = \mu$ and $|\mu|\geq K$, then $g|_{\mu} = s(\mu)$.
\end{dfn}

Let us see that this notion of regularity is equivalent to that of \cite[Definition 4.1]{BBGHSW} for self-similar groupoids.

\begin{prp}
    Let $(G,E)$ be a self-similar groupoid such that $E$ has no sources. Then, $(G,E)$ is regular if and only if for every $y  \in E^{+\infty}$ such that $g \cdot y = y$, there exists $\mu $ in
$E^*$ such that $y \in \rZ{\mu}$, $g \cdot \mu = \mu$ and $g|_\mu = s(\mu)$.
\end{prp}
\begin{proof}
The ``only if'' direction is immediate, and the ``if'' direction follows from \cite[Lemma~4.4]{BBGHSW}
\end{proof}

We recall from \cite[Section 3]{BBGHSW} (c.f \cite[Chapter 3]{Nekrashevych:Self-similar}) the construction of the limit space from a self-similar groupoid. Let $(G, E)$ be a self-similar groupoid. For $\mu,\nu \in E^{-\infty}$, we say $\mu$ is \textit{asymptotically equivalent} to $\nu$ if there is a finite set $F\subseteq G$ and a sequence $(g_{n})_{n<0}\subseteq F$ such that $d(g_{n}) = r(\mu_{n})$ and $g_{n}\cdot \mu_{n}...\mu_{-1} = \nu_{n}...\nu_{-1}$ for all $n < 0$. We write $\mu \sim_{ae} \nu$. It is shown in \cite[Section 3]{BBGHSW} that $\sim_{ae}$ is an equivalence relation and $\mu \sim_{ae} \nu$ implies $\sigma(\mu)\sim_{ae}\sigma(\nu)$. The quotient space $E^{-\infty}/\sim_{ae}$ is called the \textit{limit space of $(G,E)$} and is denoted $\Jj_{G, E}$. The induced continuous mapping from $(\sigma, E^{-\infty})$ is called the \textit{shift on $\Jj_{G, E}$} and is denoted $\tilde{\sigma}$.

It is shown in \cite[Theorem~4.3]{BBGHSW} that if $(G, E)$ is a contracting and regular self-similar groupoid such that $E$ has no sources, then $(\tilde{\sigma}, \Jj_{G,E})$ is an open, surjective, and positively expansive local homeomorphism.

In many ways, the remainder of this paper is dedicated to understanding, in various contexts, this limit space dynamical system $(\tilde{\sigma}, \Jj_{G,E})$ and the conditions above on $(G,E)$ which give rise to its regularity properties.

\section{Katsura-Exel-Pardo groupoid actions on directed graphs}\label{Katsura groupoids}

The main examples of self-similar groupoid actions on graphs that we are interested in are the Katsura-Exel-Pardo groupoid actions \cite{Katsura:class_II}. Katsura developed a family of Cuntz-Pimsner algebras using two matrices as models for Kirchberg algebras. Exel and Pardo \cite[Section 18]{Exel-Pardo:Self-similar} realised these as self-similar groupoids acting on a graph \cite[Example 7.7]{Laca-Raeburn-Ramagge-Whittaker:Equilibrium}. For Katsura-Exel-Pardo actions we completely characterise when these are contracting and regular, which turns out to be rather subtle.

Let $N \in \mathbb{N}$. A \emph{Katsura pair} is a matrix $A=(A_{ij})$ in $M_N(\mathbb{N})$ and a matrix $B=(B_{ij})$ such that $A_{ij}=0$ implies $B_{ij}=0$. Then $A$ is the adjacency matrix of the graph $E_A$ with
\[
E_A^0=\{1,2,\cdots,N\},\ E_A^1=\{e_{i,j,m}\mid 0\leq m<A_{ij}\},\ r(e_{i,j,m})=i,\
s(e_{i,j,m})=j.
\]
Exel and Pardo \cite{Exel-Pardo:Self-similar} realised that a Katsura pair gives a self-similar action on a graph in the following way. Define a group action $\sigma:\mathbb{Z}\times E_{A}\to E_{A}$ and a one-cocycle $\varphi:\mathbb{Z}\times E_{A}\to \mathbb{Z}$ as follows: write $g\in \mathbb{Z}$ multiplicatively as $g = a^{k}$, $k\in\mathbb{Z}$. Then, $\sigma(a^{k},e_{i,j,m}) = e_{i,j,\hat{m}}$ and $\varphi(a^{k},e_{i,j,m}) = a^{\hat{k}}$, where
\begin{equation}
\label{ExelPardoKatsura_action}
kB_{ij}+m=\hat{k}{A_{ij}}+\hat{m} \text{ and }0\leq \hat{m}< A_{ij}.
\end{equation}
We then obtain an action-restriction pair for ($\mathbb{Z}\times E_{A}^{0}, E_{A})$, defined by 
\begin{equation}\label{eq:Katsura_action-restriction}
(a_{i}^{k},e_{i,j,m})\to (e_{i,j,\hat{m}},a^{\hat{k}}_{j}),
\end{equation}
and we call these \emph{Katsura-Exel-Pardo groupoid actions (KEP-actions)}. See \cite[Appendix~A]{Laca-Raeburn-Ramagge-Whittaker:Equilibrium} for a more careful treatment of these actions and \cite[Example 7.7]{Laca-Raeburn-Ramagge-Whittaker:Equilibrium} for a description of the faithful quotient. We reserve the notation $(G_{B},E_A)$ for the corresponding \emph{faithful} KEP-action associated with a Katsura pair of matrices $A$ and $B$.

These actions realise their importance within $C^*$-algebras due to the following.

\begin{thm}[{\cite[Propositions 2.6, 2.9, and 2.10, Remark 2.8]{Katsura:class_II},\cite[Remark 18.3]{Exel-Pardo:Self-similar}}]\label{Kirchberg}
Let $N \in \mathbb{N}$ and let $A=(A_{ij})$ be a matrix in $M_N(\mathbb{N})$, and $B=(B_{ij})$ a matrix in $M_N(\mathbb{Z})$ such that $A$ has no zero rows and $A_{ij}=0$ implies $B_{ij}=0$. Then, the Cuntz-Pimsner algebra of the associated self-similar groupoid action $\Oo(G_{B},E_A)$ is separable, nuclear, and in the UCT class. The $K$-theory groups of $\Oo(G_{B},E_A)$ are given by
\begin{align*}
K_0(\Oo(G_{B},E_A))&= \text{coker}(I-A) \oplus \text{ker}(I-B) \text{ and } \\ 
K_1(\Oo(G_{B},E_A))&= \text{coker}(I-B) \oplus \text{ker}(I-A).
\end{align*}
Moreover, if $A$ and $B$ also satisfy
\begin{itemize}
\item $A$ is irreducible and $A_{ij}=0\Longrightarrow B_{ij}=0$ and
\item $A_{ii} \geq 2$ and $B_{i,i}=1$ for every $1 \leq i \leq N$,
\end{itemize}
then $\Oo(G_{B},E_A)$ is a unital Kirchberg algebra.
\end{thm}

Theorem \ref{Kirchberg} outlines several restrictions that can be put on a KEP-action whose associated Cuntz-Pimsner algebra is a unital Kirchberg algebra. We now consider restrictions that arise on the self-similar groupoid side. It will be helpful to first understand the kernel of a KEP-action $\phi_{A,B}:\mathbb{Z}\times E^{0}_{A}\to \PI(E^*_{A})$.

Following Exel and Pardo \cite[p. 1124]{Exel-Pardo:Self-similar}, for $\mu \in E_A^n$, write $\mu=e_{i_0,i_1,r_1}e_{i_1,i_2,r_2} \ldots e_{i_{n-1},i_n,r_n}$. Define
\begin{equation}\label{path in matrix}
A_{\mu} \coloneqq \prod_{t=0}^{n-1} A_{i_t i_{t+1}} \quad \text{ and } \quad B_{\mu}\coloneqq \prod_{t=0}^{n-1} B_{i_t i_{t+1}}.
\end{equation}

That is, $A_{\mu}$ and $B_{\mu}$ are the products of the number of edges through the set of vertices specified by $\mu$.

\begin{prp}
    \label{Katsura_Kernel}
Let $(A, B)$ be a Katsura pair. Then $a^{k}_{i}\in \text{ker}(\phi_{A,B})$ if and only if $\frac{kB_{\mu}}{A_{\mu}}\in\mathbb{Z}$ for all $\mu\in iE_{A}^{*}$.
\begin{proof}
One sees by induction on $n\in\mathbb{N}$ that for $k\in\mathbb{Z}$ and $\mu\in iE_{A}^{n}$, $a^{k}_{i}\cdot\mu = \mu$ if and only if $k\frac{B_{\mu[1,i]}}{A_{\mu[1,i]}}\in\mathbb{Z}$ for all $i\leq n$.
Therefore, $a^{k}_{i}\in \text{ker}(\phi_{A,B})$ if and only if $k\frac{B_{\mu}}{A_{\mu}}\in\mathbb{Z}$ for all $\mu\in iE^{*}_{A}$.
\end{proof}
\end{prp}

Note that for $\mu\in E^{n}$ such that $B_{\mu}\neq 0$, $k\frac{B_{\mu}}{A_{\mu}}\in\mathbb{Z}$ if and only if $\frac{A_{\mu}}{\text{gcd}(A_{\mu}, |B_{\mu}|)}$ divides $k$. So, by Proposition \ref{Katsura_Kernel}, the group $(G_{B})_{i}$ is finite if and only if $\text{max}_{\mu\in iE^{*}: B_{\mu}\neq 0}\frac{A_{\mu}}{\text{gcd}(A_{\mu}, |B_{\mu}|)} <\infty$ and its cyclic order $o_{i}$ is the smallest $k\in\mathbb{Z}$ such that $\frac{A_{\mu}}{\text{gcd}(A_{\mu},|B_{\mu}|)}$ divides $k$, for all $B_{\mu}\neq 0$. Thus, $o_{i} = \text{lcm}(\{\frac{A_{\mu}}{\text{gcd}(A_{\mu},|B_{\mu}|)}\mid\mu\in iE_{A}^{*}: B_{\mu}\neq 0\})$. We have the following corollary.

\begin{cor}
\label{invariance}
Suppose $(G_{B}, E_{A})$ is a KEP-action, and define 
\[
 E_{A,<\infty}^{0} \coloneqq \{i\in E^{0}_{A}\mid (G_{B})_{i} \text{ is finite.}\}.
 \] 
Then $E_{A,<\infty}^{0}$ is invariant in the sense that if $e\in E^{1}$ satisfies $B_{e}\neq 0$ and $r(e)\in E_{A,<\infty}^{0}$, then $s(e)\in E_{A,<\infty}^{0}$.
\end{cor}
\begin{proof}
Let $\nu\in E^{*}$ and $e\in E^{1}$ be such that $s(e) = r(\nu)$ and $B_{e\nu}\neq 0$. We have $A_{e\nu} = A_{e}A_{\nu}$, $|B_{e\nu}| = |B_{e}||B_{\nu}|$ and therefore  $\text{gcd}(A_{e},|B_{e}|)\text{gcd}(A_{\nu}, |B_{\nu}|)$ divides $A_{e\nu}$ and $B_{e\nu}$, so that $\text{gcd}(A_{e},|B_{e}|)\text{gcd}(A_{\nu}, |B_{\nu}|)$ divides $\text{gcd}(A_{e\nu}, |B_{e\nu}|)$. If we let $A'_{\nu} = \frac{A_{\nu}}{\text{gcd}(A_{\nu}, |B_{\nu}|)}$ and $B'_{\nu} = \frac{B_{\nu}}{\text{gcd}(A_{\nu}, |B_{\nu}|)}$, then $$m\coloneqq \frac{\text{gcd}(A_{e\nu}, |B_{e\nu}|)}{\text{gcd}(A_{e},|B_{e}|)\text{gcd}(A_{\nu}, |B_{\nu}|)} = \text{gcd}(A'_{e}A'_{\nu}, |B'_{e}||B'_{\nu}|).$$
Since $\text{gcd}(A'_{\nu}, B'_{\nu}) = 1$, the factors in $m$ that divide $B'_{\nu}$ must divide $A_{e}$ and the factors that divide $A'_{\nu}$ must divide $B_{e}$. From this, we see that $m$ divides $A_{e}B_{e}$ and therefore
$$\text{gcd}(A_{e},|B_{e}|)\text{gcd}(A_{\nu}, |B_{\nu}|)\leq \text{gcd}((A_{e\nu}, |B_{e\nu}|))\leq A_{e}|B_{e}|\text{gcd}(A_{e},|B_{e}|)\text{gcd}(A_{\nu}, |B_{\nu}|).$$
So, if we let $C = \text{max}_{e\in E^{1}}\frac{A_{e}}{\text{gcd}(A_{e},|B_{e}|)}$ and $D = \text{max}_{e\in E^{1}}|B_{e}|\text{gcd}(A_{e}, |B_{e}|)$, then
\begin{equation}
    \label{1}
    C\frac{A_{\nu}}{\text{gcd}(A_{\nu}, |B_{\nu}|)} \geq \frac{A_{e\nu}}{\text{gcd}(A_{e\nu}, |B_{e\nu}|)}\geq \frac{1}{D}\frac{A_{\nu}}{\text{gcd}(A_{\nu}, |B_{\nu}|)}.
\end{equation}

In particular,  if $\text{max}_{\nu\in s(e)E^{*}:B_{\nu}\neq 0}\frac{A_{\nu}}{\text{gcd}(A_{\nu}, |B_{\nu}|)} = \infty$, then
$$\infty = \text{max}_{\nu\in s(e)E^{*}:B_{\nu}\neq 0}\frac{A_{e\nu}}{\text{gcd}(A_{e\nu}, |B_{e\nu}|)}\leq \text{max}_{\mu\in r(e)E^{*}:B_{\mu}\neq 0}\frac{A_{\mu}}{\text{gcd}(A_{\mu}, |B_{\mu}|)}.$$ This proves that, for $B_{e}\neq 0$, if $s(e)\notin E^{0}_{A,<\infty}$, then $r(e)\notin E^{0}_{A,<\infty}$.
\end{proof}

Now, define
\[
E^{0}_{A,\infty} \coloneqq \{i\in E_{A}^{0}\mid  (G_{B})_{i} = \mathbb{Z}\},
\]
and note that $E^{0}_{A,\infty} = E^{0}_{A}\setminus E^{0}_{A, <\infty}$. Similarly, define
\[
E^{1}_{A,\infty} \coloneqq \{e\in E^{1}_{A}\mid s(e)\in E^{0}_{A,\infty}\text{ and } B_{e}\neq 0\}.
\]
By Corollary \ref{invariance}, $r(E^{1}_{A,\infty})\subseteq E^{0}_{A,\infty}$, so $E_{A,\infty} \coloneqq (E^{0}_{A,\infty}, E^{1}_{A,\infty}, r ,s)$ is a sub-graph of $E_{A}$, and the action of $G_{B,\infty} \coloneqq \{g\in G: s(g)\in E^{0}_{A,\infty}\}$ on $E^{*}_{A}$ restricts to an action $\phi_{A,B,\infty}:G_{B,\infty}\to \PI(E^{*}_{A,\infty})$.

By re-ordering the vertices if necessary, we may assume $E^{0}_{A,\infty} = \{1,..,k\}$ for some $k\leq N$. Then, letting $A_{\infty}$ be the adjacency matrix of $E^{1}_{A,\infty}$ we see that
\[
(A_{\infty})_{i,j} = \begin{cases}
A_{i,j} \text{ if }  B_{i,j}\neq 0,\\
0\text{ otherwise.}
\end{cases}
\]
Since $(G_{B})_{i}$ is infinite for $i\leq k$, we must have (by Corollary \ref{invariance}) $r^{-1}(i)\cap E_{A,\infty}^{1}\neq\emptyset$. Hence $A_{\infty}$ has no zero rows.
Letting $B_{\infty} \coloneqq (B_{i,j})_{i,j\leq k}$, we have $\phi_{A_{\infty},B_{\infty}} = \phi_{A,B,\infty}$. Note that in general $G_{B_{\infty}}$ is a quotient of $G_{B,\infty} = \mathbb{Z}\times E^{0}_{A,\infty}$.

Similarly, let 
\[
E^{1}_{A,<\infty} \coloneqq \{e\in E^{1}_{A}\mid  r(e)\in E^{0}_{A,<\infty}\text{ and }B_{e}\neq 0\},
\]
so that $E_{A,<\infty} \coloneqq (E^{1}_{A,<\infty}, E^{0}_{A,<\infty}, r, s)$ is, by Corollary \ref{invariance}, a sub-graph of $E_{A}$ and set $G_{B,<\infty} \coloneqq \{g\in G_{B}: r(g)\in E^{0}_{A,<\infty}\}$. Then, $\phi_{A,B}$ restricts to an action $\phi_{A,B,<\infty}:G_{B,<\infty}\to \PI(E^{*}_{A,<\infty})$ and letting $A_{<\infty}$ be the adjacency matrix of $E_{A,<\infty}$ and $B_{<\infty} = (B_{i,j})_{i,j>k}$, we have $\phi_{A_{<\infty}, B_{<\infty}} = \phi_{A,B,<\infty}$ and $G_{B_{<\infty}} = G_{B,<\infty}$. Note however that $E_{A_{<\infty}}$ may have sources even if $E_{A}$ has none.

\begin{dfn}
\label{dfn:infinite_finite_decomposition}
 Let $(A,B)$ be a Katsura pair. We call $(A_{\infty}, B_{\infty})$ the \emph{infinite part} of $(A,B)$ and $(A_{<\infty}, B_{<\infty})$ the \emph{finite part} of $(A,B)$, as defined immediately above.
\end{dfn}

For the following proposition, recall the notion of \emph{contracting} from Section \ref{sec: definitions}.

\begin{prp}
\label{Contracting_Katsura}
Let $(A,B)$ be a Katsura pair. Then, the KEP-action $(G_{B}, E_{A})$ is contracting if and only if $(G_{B,\infty}, E_{A_{\infty}})$ is contracting.
\end{prp}
\begin{proof}
The ``only if'' direction is immediate, so we prove the ``if'' direction by proving its contrapositive. If $(G_{B}, E_{A})$ is not contracting, then for every finite set $F$ such that $G_{B,<\infty}\cup E^{0}\subseteq F\subseteq G_{B}$, there is $g\in G_{B}$ such that $g|_{\mu_{n}}\notin F$ for infinitely many paths $(\mu_{n})_{n\in\mathbb{N}}$. For $\mu\in E^{*}_{A}$ such that $B_{\mu} = 0$, we have $g|_{\mu} = s(\mu)\in F$ and hence $B_{\mu_{n}}\neq 0$. If $\mu\in E^{n}_{A}$ satisfies $B_{\mu}\neq 0$ and $r(\mu_{i})\in E^{0}_{A, <\infty}$ for some $i\leq n$, then by Corollary \ref{invariance}, we have $s(\mu)\in E^{0}_{A, <\infty}$, and therefore $h|_{\mu}\in G_{B, <\infty}\subseteq F$ for any $h\in r(\mu)G_{B}$. Since $g|_{\mu_{n}}\notin F$ and $B_{\mu_{n}}\neq 0$, it follows that $\mu_{n}\in E_{A,\infty}^{*}$ for all $n\in\mathbb{N}$ and $g\in G_{B,\infty}$. Therefore, $(G_{B_{\infty}}, E_{A_{\infty}})$ is not contracting.
\end{proof}
Now, we determine a necessary and sufficient condition for the infinite part of a KEP-action to be contracting. Suppose $G$ is a finitely generated groupoid with generating set $S = S^{-1}$, with associated length function $\ell_S:G \to \mathbb{N}$, and suppose $(G,E)$ is an action-restriction pair. Following Nekrashevych \cite[Definition 2.11.9]{Nekrashevych:Self-similar}, the \emph{contraction coefficient} of $(G,E)$ is the quantity:
\begin{equation}\label{contraction_coeff}
\rho = \limsup_{n \to \infty}\left( \limsup_{g \in G, \ell_S(g) \to \infty} \max_{\mu \in d(g)E^n} \frac{\ell_S(g|_{\mu})}{\ell_S(g)}  \right)^{1/n}.
\end{equation}

Nekrashevych proved the following result in the case of a self-similar group action. However, his proof goes through line-for-line with the obvious extension from words to paths in the action-restriction pair setting.

\begin{prp}[{\cite[Lemma 2.11.10 and Proposition 2.11.11]{Nekrashevych:Self-similar}}]\label{Cont_coeff_prp}
Let $G$ be a finitely generated groupoid and $E$ a finite graph with no sources. If $(G,E)$ is an action-restriction pair, then the contraction coefficient $\rho$ is finite and does not depend on the generating set. Furthermore, $(G,E)$ is contracting if and only if $\rho <1$.
\end{prp}

Now suppose that $(\mathbb{Z}\times E^{0}_{A},E_A)$ is the action restriction pair associated to a Katsura pair $(A,B)$. Then \eqref{contraction_coeff} becomes:
\begin{equation}\label{contraction_coeff_Katsura}
\rho = \limsup_{n \to \infty}\left( \limsup_{m \to \infty} \max_{\mu \in E_A^n} \frac{\ell_S(a^m_{r(\mu)}|_{\mu})}{m}  \right)^{1/n}
\end{equation}

\begin{prp}\label{Katsura_contracting}
Let $(A,B)$ be a Katsura pair such that $A$ has no zero rows and let $(\mathbb{Z}\times E^{0}_{A}, E_{A})$ be the associated action-restriction pair. Then the contraction coefficient is given by
\[
\rho=\limsup_{n \to \infty}\left( \max_{\mu \in E_A^n} \frac{|B_{\mu}|}{A_{\mu}} \right)^{1/n}.
\]
\end{prp}
\begin{proof}
We first prove by induction on $m \in \NN$ that, for fixed $1 \leq i,j \leq N$ and $0\leq r <N$, that
\begin{equation}\label{induc on m}
a_i^{m} \cdot e_{i,j,r}\nu=e_{i,j,r_m}(a_j^{l} \cdot \nu) \quad \text{ where } l=m\frac{B_{ij}}{A_{ij}} + \frac{r-r_m}{A_{ij}}, \quad \text{ for all } \nu \in jE^*.
\end{equation}
For $m=1$, using \eqref{ExelPardoKatsura_action} we have $B_{ij}+r=l_1 A_{ij} +r_1$ so that
\[
a_i \cdot e_{i,j,r}\nu=e_{i,j,r_1}(a_j^{l} \cdot \nu) \quad \text{ where } l=l_1=\frac{B_{ij}}{A_{ij}} + \frac{r-r_1}{A_{ij}},
\]
as desired. Using \eqref{induc on m} for $m-1$ we have
\begin{equation}\label{induc on m_1}
a_i^{m} \cdot e_{i,j,r}\nu=a_i \cdot e_{i,j,r_{m-1}}(a_j^{l'} \cdot \nu) \quad \text{ where } l'=(m-1)\frac{B_{ij}}{A_{ij}} + \frac{r-r_{m-1}}{A_{ij}}.
\end{equation}
From \eqref{ExelPardoKatsura_action} we have $B_{ij}+r_{m-1}=l_m A_{ij} +r_m$, so that \eqref{induc on m_1} gives
\[
a_i^{m} \cdot e_{i,j,r}\nu=e_{i,j,r_{m}}(a_j^{l_m} \cdot (a_j^{l'} \cdot \nu))=e_{i,j,r_{m}}(a_j^{l'+l_m} \cdot \nu),
\]
where
\[
 l=l'+l_m=(m-1)\frac{B_{ij}}{A_{ij}} + \frac{r-r_{m-1}}{A_{ij}}+\frac{B_{ij}}{A_{ij}} + \frac{r_{m-1}-r_m}{A_{ij}}=m\frac{B_{ij}}{A_{ij}} + \frac{r-r_{m}}{A_{ij}}.
\]
Thus, \eqref{induc on m} holds.

Suppose $\mu=e_{i_0,i_1,r_1}e_{i_1,i_2,r_2} \ldots e_{i_{n-1},i_n,r_n}$, $\nu=e_{i_0,i_1,r'_1}e_{i_1,i_2,r'_2} \ldots e_{i_{n-1},i_n,r'_n}$, and~$a_{i_0}^m \cdot \mu=\nu$. We prove by induction on $|\mu| = n$ that
\begin{equation}\label{induc on n}
a_{r(\mu)}^{m}|_{\mu}=a_{s(\mu)}^{l}\quad \text{ where } l=m\frac{B_{\mu}}{A_{\mu}} + \sum_{t=1}^n (r_t-r'_t) \frac{B_{\mu[t+1,n]}}{A_{\mu[t,n]}}.
\end{equation}
For $|\mu|=1$, \eqref{induc on n} holds by \eqref{induc on m}. Using \eqref{induc on n} for $n-1$ and \eqref{induc on m}, we compute for $|\mu|=n$:
\begin{align*}
l&=\left( m\frac{B_{\mu[1,n-1]}}{A_{\mu[1,n-1]}} + \sum_{t=1}^{n-1} (r_t-r'_t) \frac{B_{\mu[t+1,n-1]}}{A_{\mu[t,n]}} \right)\frac{B_{i_{n-1}i_{n}}}{A_{i_{n-1}i_{n}}} + \frac{r_n-r'_n}{A_{i_{n-1}i_n}} \\
&=m\frac{B_{\mu}}{A_{\mu}} + \sum_{t=1}^{n} (r_t-r'_t) \frac{B_{\mu[t+1,n]}}{A_{\mu[t,n]}},
\end{align*}
so \eqref{induc on n} holds by induction.

We now use \eqref{induc on n} to compute the contraction coefficient
\begin{align*}
\rho &= \limsup_{n \to \infty}\left( \limsup_{m \to \infty} \max_{\mu \in E_A^n} \frac{\ell_S(a^m_{r(\mu)}|_{\mu})}{m}  \right)^{1/n}\\
&= \limsup_{n \to \infty}\left( \limsup_{m \to \infty} \max_{\mu \in E_A^n} \left|\frac{B_{\mu}}{A_{\mu}} + \frac{1}{m}\left( \sum_{t=1}^n (r_t-r'_t) \frac{B_{\mu[t+1,n]}}{A_{\mu[t,n]}} \right)\right|\right)^{1/n} \\
&= \limsup_{n \to \infty}\left( \max_{\mu \in E_A^n} \frac{|B_{\mu}|}{A_{\mu}} \right)^{1/n}. \qedhere
\end{align*}
\end{proof}

\begin{cor}
\label{contracting_characterization}
Let $(A,B)$ be a Katsura pair. Then, the associated KEP-action $(G_{B}, E_{A})$ is contracting if and only if $\limsup_{n \to \infty}\left( \max_{\mu \in E_{A,\infty}^{n}} \frac{|B_{\mu}|}{A_{\mu}} \right)^{1/n} < 1$.
\end{cor}

For the following proposition, recall the notion of \emph{regular} from Section \ref{sec: definitions}.

\begin{prp}
\label{regulardecomposition}
Let $(A,B)$ be a Katsura pair. Then, the KEP-action $(G_{B}, E_{A})$ is regular if and only if $(G_{B,\infty}, E_{A_{\infty}})$ and $(G_{B_{<\infty}}, E_{A_{<\infty}})$ are regular.
\end{prp}
\begin{proof}
The ``only if'' direction is immediate, so we prove the ``if'' direction. Suppose $g\in G$. we show there is $M\in\mathbb{N}$ such that $g\cdot \mu = \mu$, $g|_{\mu}\neq s(\mu)$ implies $|\mu|\leq M$.

Since $G_{B_{<\infty}}$ is finite, there is an $M'\in\mathbb{N}$ such that for every $h\in G_{B_{<\infty}}$ and $\nu\in E^{*}_{A,<\infty}$ satisfying $h\cdot \nu = \nu$ and $g|_{\nu}\neq s(\nu)$, we have $|\nu|\leq M'$.

Note that $g|_{\mu}\neq s(\mu)$ implies $B_{\mu}\neq 0$, so by Corollary \ref{invariance}, we can write $\mu = \mu_{1}e\mu_{2}$ for some $\mu_{1}\in E^{*}_{A,\infty}$ and $\mu_{2}\in E^{*}_{A,<\infty}$ (in this decomposition we allow $\mu_{1} = \emptyset$ or $\mu_{2}=\emptyset$).

If $\mu_{1} = \emptyset$, then $|\mu|\leq M' + 1$.

If $\mu_{1}\neq\emptyset$, then $g\in G_{B,\infty}$. Let $M''$ be such that for $\nu\in E^{*}_{A,\infty}$, $g\cdot\nu = \nu$ and $g|_{\nu}\neq s(\nu)$ implies $|\nu|\leq M''$. Then, $|\mu|\leq M'' + M' + 1$.

In either case, we have $|\mu|\leq M\coloneqq M' +M'' +1.$
\end{proof}

\begin{prp}\label{Katsura_regular}
Let $(A,B)$ be a Katsura pair such that $A$ has no zero rows. If the associated action-restriction pair $(\mathbb{Z}\times E^{0}_{A},E_A)$ is contracting, then it is regular.
\end{prp}
\begin{proof}
Suppose $(\mathbb{Z}\times E^{0},E_A)$ is contracting. Write $g = a^{k}_{i}$. By Proposition \ref{Cont_coeff_prp} and Proposition \ref{Katsura_contracting}, there is $M\in\mathbb{N}$ such that $\frac{|B_{\mu}|}{A_{\mu}} < \frac{1}{k}$ for all $|\mu|\geq M$ and hence $k\frac{B_{\mu}}{A_{\mu}}\notin \mathbb{Z}$. It follows that $g\cdot\mu\neq \mu$ for $|\mu|\geq M$. Hence, $(\mathbb{Z}\times E^{0},E_A)$ is regular.
\end{proof}
\begin{cor}
\label{regular_and_contracting_decomposition}
    Let $(A,B)$ be a Katsura pair such that the KEP-action $(G_{B}, E_{A})$ is contracting. Then, $(G_{B}, E_{A})$ is regular if and only if $(G_{B_{<\infty}}, E_{A_{<\infty}})$ is regular.
\end{cor}
\begin{proof}
Follows immediately from Proposition \ref{regulardecomposition} and Proposition \ref{Katsura_regular}.
\end{proof}

We now determine a necessary and sufficient condition for when the finite part of a KEP-action is regular.

\begin{prp}
\label{finite_regular}
Let $(A,B)$ be a Katsura pair such that $\text{sup}_{\mu\in E^{*}:B_{\mu}\neq 0}\frac{A_{\mu}}{\text{gcd}(A_{\mu}, |B_{\mu}|)} <\infty$. 
Then, the KEP-action $(G_{B}, E_{A})$ is regular if and only if there is $K\in\mathbb{N}$ such that for all $\mu\in E^{K}_{A}$ and $\omega\in s(\mu)E^{*}_{A}$, $A_{\omega}$ divides $\frac{B_{\mu}}{\text{gcd}(A_{\mu}, |B_{\mu}|)}B_{\omega}$.
\end{prp}
\begin{proof}
    We first prove the ``only if'' direction. Since $\text{sup}_{\mu\in E^{*}:B_{\mu}\neq 0}\frac{A_{\mu}}{\text{gcd}(A_{\mu}, |B_{\mu}|)} <\infty$, $G_{B}$ is finite, so by regularity there is $K\in\mathbb{N}$ such that for every $g\in G_{B}$ and $\mu\in E^{K}_{A}$, if $g\cdot \mu = \mu$, then $g|_{\mu} = s(\mu)$. In particular, for $\mu\in E^{K}_{A}$ and $k =  \frac{A_{\mu}}{\text{gcd}(A_{\mu}, |B_{\mu}|)}$ we have $a^{k}_{r(\mu)}\cdot \mu = \mu$ and therefore $a^{k}_{r(\mu)}|_{\mu} = s(\mu)$, which implies $\frac{B_{\mu}}{\text{gcd}(A_{\mu}, |B_{\mu}|)}\cdot \frac{B_{\omega}}{A_{\omega}} = \frac{A_{\mu}}{\text{gcd}(A_{\mu}, |B_{\mu}|)}\cdot \frac{B_{\mu\omega}}{A_{\mu\omega}}\in\mathbb{Z}$ for all $\omega\in s(\mu)E_{A}^{*}$. Therefore, $A_{\omega}$ divides $\frac{B_{\mu}}{\text{gcd}(A_{\mu}, |B_{\mu}|)}B_{\omega}$ for all $\mu\in s(\mu)E_{A}^{*}$.

     We prove the ``if'' direction now. Suppose $g = a_{i}^{k}$ satisfies $a_{i}^{k}\cdot \mu = \mu$ for some $\mu\in iE_{A}^{K}$. This implies $\frac{kB_{\mu}}{A_{\mu}}\in\mathbb{Z}$, which is equivalent to $\frac{A_{\mu}}{\text{gcd}(A_{\mu}, |B_{\mu}|)}$ divides $k$. Let $m\in\mathbb{Z}$ be such that $k = m\cdot \frac{A_{\mu}}{\text{gcd}(A_{\mu}, |B_{\mu}|)}$. By the hypothesis, if $\omega\in s(\mu)E^{*}_{A}$, then $\frac{A_{\mu}}{\text{gcd}(A_{\mu}, |B_{\mu}|)}\cdot \frac{B_{\mu\omega}}{A_{\mu\omega}} = \frac{B_{\mu}}{\text{gcd}(A_{\mu}, |B_{\mu}|)}\cdot \frac{B_{\omega}}{A_{\omega}}\in\mathbb{Z}$. Hence, $k\cdot\frac{B_{\mu\omega}}{A_{\mu\omega}} = m\cdot \frac{A_{\mu}}{\text{gcd}(A_{\mu}, |B_{\mu}|)}\cdot \frac{B_{\mu\omega}}{A_{\mu\omega}}\in\mathbb{Z}$. Then, letting $l = \frac{k B_{\mu}}{A_{\mu}}$, by Proposition \ref{Katsura_Kernel}, we have $a_{i}^{k}|_{\mu} = a^{l}_{s(\mu)}= s(\mu)$.
\end{proof}

We summarise the results of this section into a theorem.

\begin{thm}
Let $(A,B)$ be a Katsura pair. Then, the KEP-action $(G_{B}, E_{A})$ is contracting and regular if and only if $\limsup_{n \to \infty}\left( \max_{\mu \in E_{A,\infty}^{n}} \frac{|B_{\mu}|}{A_{\mu}} \right)^{1/n} < 1$ and there is $K\in\mathbb{N}$ such that $A_{\omega}$ divides $\frac{B_{\mu}}{\text{gcd}(A_{\mu}, |B_{\mu}|)}B_{\omega}$
for all $\mu\in E^{K}_{A, <\infty}$ and $\omega\in s(\mu)E^{*}_{A,<\infty}$
\end{thm}

Later in this paper, we will restrict to considering Katsura pairs with $B$ taking values either $0$ or $1$. The above 
theorem has a nice reformulation in this setting.

\begin{cor}
\label{regular_characterization_01}
Let $(A, B)$ be a Katsura pair such that $B\in M_{N}(\{0,1\})$. If the KEP-action $(G_{B}, E_{A})$ is regular, then $(G_{B}, E_{A})$ is contracting. Moreover, $(G_{B}, E_{A})$ is regular if and only if there is $K\in\mathbb{N}$ such that 
\begin{itemize}
    \item $\mu\in E^{*}_{A,\infty},\text{ } A_{\mu} = 1\implies |\mu|\leq K$ and
    \item $\mu\in E^{*}_{A, <\infty}$, $|\mu|\geq K\implies A_{\mu[K, |\mu|]} = 1$.
\end{itemize}
\end{cor}
\begin{proof}
By Proposition \ref{Contracting_Katsura}, it suffices to show $(G_{B,\infty}, E_{A,\infty})$ is contracting. By Proposition \ref{regulardecomposition}, $(G_{B,\infty}, E_{A,\infty})$ is regular. For $\mu\in E^{*}_{A,\infty}$, we have $a_{r(\mu)}\cdot \mu = \mu$ if and only if $A_{\mu} = 1$, in which case $a_{r(\mu)}|_{\mu} = a_{s(\mu)}$. By regularity there is $K\in\mathbb{N}$ such that $A_{\mu} = 1$ implies $|\mu|\leq K$. Hence, for $\mu\in E^{*}_{A,\infty}$ we have $A_{\mu}\geq 2^{\frac{|\mu|}{K}}$, so that the contracting coefficient (by Proposition \ref{Katsura_contracting}) is $\rho < (\frac{1}{2})^{\frac{1}{K}} < 1$. By proposition \ref{Cont_coeff_prp}, $(G_{B,\infty}, E_{A,\infty})$ is contracting.

The second point is equivalent (by Proposition \ref{finite_regular}) to regularity of $(G_{B_{<\infty}}, E_{A_{<\infty}})$.

To prove that the two bullet points imply $(G_{B}, E_{A})$ is regular, it therefore suffices to show (by Proposition \ref{regulardecomposition} and \ref{Katsura_regular}) that $(G_{B,\infty},E_{A,\infty})$ is contracting, but this is the same argument as above.
\end{proof}

We use our characterisations of contracting and regular to provide four examples of Katsura pairs that exhibit all the possible combinations of the two properties holding (or not holding). 

\begin{center}
\begin{figure}
\begin{tikzpicture}[scale=1.0]
\node at (0,0) {$2$};
\node[vertex] (vertexe) at (0,0)   {$\quad$}
	edge [->,>=Straight Barb,out=-30,in=30,loop] node[left,pos=0.5]{$\scriptstyle e_{2,2,0}$} (vertexe)
	edge [->,>=Straight Barb,out=-55,in=55,loop,looseness=10] node[right,pos=0.5]{$\scriptstyle e_{2,2,1}$} (vertexe);
\node at (-3,0) {$1$};
\node[vertex] (vertex-a) at (-3,0)   {$\quad$}
	edge [->,>=Straight Barb,out=35,in=145] node[below,swap,pos=0.5]{$\scriptstyle e_{2,1,0}$} (vertexe)
	edge [->,>=Straight Barb,out=210,in=150,loop] node[right,pos=0.5]{$\scriptstyle e_{1,1,1}$} (vertex-a)
	edge [->,>=Straight Barb,out=235,in=125,loop,looseness=10] node[left,pos=0.5]{$\scriptstyle e_{1,1,0}$} (vertex-a)
	edge [<-,>=Straight Barb,out=310,in=230] node[below,swap,pos=0.5]{$\scriptstyle e_{1,2,0}$} (vertexe);
\end{tikzpicture}
\caption{The graph $E_A$ specified by the adjacency matrix $A$ from Examples \ref{example1} and \ref{example2}.}
\label{Katsura_graph_rmkexm}
\end{figure}
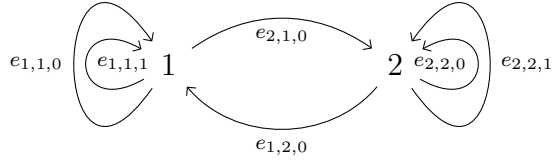
\end{center}

\begin{exm}[Contracting and regular]\label{example1}
Let $(G_B,E_A)$ be the KEP-action defined by 
\begin{equation}\label{Katsura_matrices_rmkexm}
A=\left(\begin{matrix} 2 & 1 \\ 1 & 2 \end{matrix}\right) \quad \text{ and } \quad B=\left(\begin{matrix} 1 & 1 \\ 0 & 1 \end{matrix}\right).
\end{equation}
Then $A$ is the adjacency matrix for the graph $E_A$ depicted in Figure \ref{Katsura_graph_rmkexm}.

Using the relations in \eqref{ExelPardoKatsura_action}, for $\mu \in 1 E_A^*$ and $\nu \in 2E_A^*$, we obtain the self-similar action defined by the partial isomorphisms
\begin{align*}
a_1\cdot e_{1,1,0}\mu=e_{1,1,1}\mu; \quad\quad& a_2 \cdot e_{2,1,0}\mu=e_{2,1,0}\mu;\\
a_1\cdot e_{1,1,1}\mu=e_{1,1,0}(a_1 \cdot \mu); \quad\quad& a_2 \cdot e_{2,2,0}\nu=e_{2,2,1}\nu; \\
a_1\cdot e_{1,2,0}\nu=e_{1,2,0}(a_2 \cdot \nu);\quad\quad& a_2 \cdot e_{2,2,1}\nu=e_{2,2,0}(a_2 \cdot \nu);
\end{align*}

We have $A_{\infty} = \left(\begin{matrix} 2 & 1 \\ 0 & 2 \end{matrix}\right)$ and $A_{<\infty} = \emptyset$. Therefore, being contracting and regular depends only on the first bullet point in Corollary \ref{regular_characterization_01} holding. For $\mu\in E_{A_{\infty}}$, we have $|\mu|>1$ implies $A_{\mu} > 1$. Hence, $(G_{B}, E_{A})$ is contracting and regular. Note also that $\max_{\mu \in E_{A,\infty}^n} \frac{|B_{\mu}|}{A_{\mu}} = \frac{1}{2^{n-1}}$. Therefore, the contracting coefficient
\[
\rho=\limsup_{n \to \infty}\left( \max_{\mu \in E_{A,\infty}^n} \frac{|B_{\mu}|}{A_{\mu}} \right)^{1/n}=\limsup_{n \to \infty} \left( \frac{1}{2^{n-1}} \right)^{1/n} = \frac{1}{2}.
\]
\end{exm}

\begin{exm}[Not contracting and not regular]\label{example2}
Let $(G_B,E_A)$ be the KEP-action defined by 
\begin{equation}\label{Katsura_matrices_rmkexm2}
A=\left(\begin{matrix} 2 & 1 \\ 1 & 2 \end{matrix}\right) \quad \text{ and } \quad B=\left(\begin{matrix} 1 & 1 \\ 1 & 1 \end{matrix}\right).
\end{equation}
Since $A$ is the same as in Example \ref{example1}, $A$ is again the adjacency matrix for the graph $E_A$ in Figure~\ref{Katsura_graph_rmkexm}. The action is also the same other than $a_1 \cdot e_{1,2,0}\nu=e_{1,2,0}\nu$.

However, we have $A_{\infty} = A$  and observe that $\max_{\mu \in E_A^n} \frac{|B_{\mu}|}{A_{\mu}} = 1$. Therefore, the contracting coefficient is $\rho=1$ so that $(G_B,E_A)$ is not contracting.  It follows from Corollary \ref{regular_characterization_01} that  $(G_B,E_A)$ is not regular.\qed
\end{exm}

\begin{exm}[Contracting and not regular]\label{example3.3}
Let $(G_B,E_A)$ be the KEP-action defined by 
\begin{equation}\label{Katsura_matrices_rmkexm3.3}
A=\left(\begin{matrix} 1 & 2 \\ 2 & 1 \end{matrix}\right) \quad \text{ and } \quad B=\left(\begin{matrix} 1 & 1 \\ 0 & 1 \end{matrix}\right).
\end{equation}
We have $A_{<\infty} = \left(\begin{matrix} 1 & 2 \\ 0 & 1 \end{matrix}\right)$ and $A_{\infty} = \emptyset$. Therefore, $G_{B}$ is finite, making $(G_{B}, E_{A})$ automatically contracting. However, for every $k\in\mathbb{N}$, we have $B_{\mu_{k}} = 1$ and $A_{\mu_{k}} = 2$ for $\mu_{k} = e_{1,1,0}^{k-1}e_{1,2,0}$. Corollary \ref{regular_characterization_01} implies that $G_{B}$ is not regular.
\end{exm}

The three examples above have $B\in M_{N}(\{0,1\})$. For an example to be not contracting and regular, by Corollary \ref{regular_characterization_01}, we must have $B\notin M_{N}(\{0,1\})$.

\begin{exm}[Not contracting but regular]\label{example4}
Let $(G_B,E_A)$ be the KEP-action defined by 
\begin{equation}\label{Katsura_matrices_rmkexm4}
A=\left(\begin{matrix} 2 \end{matrix}\right) \quad \text{ and } \quad B=\left(\begin{matrix} 3 \end{matrix}\right).
\end{equation}
We have $A_{\infty} = \left(\begin{matrix} 2 \end{matrix}\right)$ with contracting coefficient $\rho = \frac{3}{2} > 1$. Hence, $(G_{B}, E_{A})$ is not contracting. Moreover, for every $\mu\in E_{A}^{*}$, we have $\frac{B_{\mu}}{A_{\mu}} = \left(\frac{3}{2}\right)^{|\mu|}$. Therefore, for every $k\in\mathbb{Z}$, $|\mu| > |k|$ implies $\frac{kB_{\mu}}{A_{\mu}}\notin\mathbb{Z}$ and hence $a^{k}\cdot \mu\neq \mu.$ It follows that $(G_{B}, E_{A})$ is regular.
\end{exm}

\section{Binary factors of shifts of finite type and self-similar groupoids}\label{Putnam's construction}
In this section we show a certain class of topological dynamical systems introduced by Putnam in \cite{Putnam:Binary} can be realised as shifts on the limit spaces of contracting and regular self-similar groupoids. We would like to note that a generalisation of the spaces Putnam's systems are defined on appears in the Ph.D. thesis of Haslehurst, see \cite[Chapter~3]{Haslehurst}. We will first recall Putnam's construction.

Our notation differs slightly from \cite{Putnam:Binary}. First, Putnam uses $G$ to denote a directed graph, whereas we have reserved $G$ for groupoids and $E$, $H$ for graphs. He also calls the source map $s$ the \emph{initial map} and denotes it $i$, and calls the range map $r$ the \emph{terminal map}, denoting it $t$. The last important distinction is his notation for infinite paths. We write an infinite path $x = (...x_{-2}, x_{-1})\in E^{-\infty}$, whereas Putnam writes $(x_{1},x_{2},...)\in X^{+}_{E}$ for the same path.

\subsection{Putnam's construction}\label{Putnam's construction_subsec}
Let $E$ and $H$ be finite directed graphs, and let $\xi = \xi^{0},\xi^{1}:H \to E$ be two injective graph homomorphisms (embeddings) satisfying $\xi^{0}|_{H^{0}} = \xi^{1}|_{H^{0}}$ and $\xi^{0}(H^{1})\cap \xi^{1}(H^{1})=\emptyset$. We will refer to $\xi$ satisfying the properties above as an \textit{embedding pair}.

For $\mu,\nu\in E^{-\infty}$ with $\mu = \ldots \mu_{-2}\mu_{-1}$ and $\nu = \ldots \nu_{-2}\nu_{-1}$, we will say $\mu \sim_{\xi} \nu$ if $\mu = \nu$, or there is $n<0$, $i\in\{0,1\}$, and $(y_{k})_{k\leq n}\subseteq H^{1}$ such that $\mu_{k} = \xi^{i}(y_{k})$ and $\nu_{k} = \xi^{1-i}(y_{k})$, for all $k\leq n$, and one of the following hold:
\begin{enumerate}
    \item $n=-1$,
    \item $n < 1$, $\mu_{j} = \nu_{j}$ for all $j > n+1$, and there is $y_{n+1}\in H^{1}$ such that $\mu_{n+1} = \xi^{1-i}(y_{n+1})$ and $\nu_{n+1} = \xi^{i}(y_{n+1})$, or
    \item $n<1$, $\mu_{j} = \nu_{j}$ for all $j\geq n+1$, and $\mu_{n+1}= \nu_{n+1}\notin H_{\xi}^{1}$.
\end{enumerate}
By \cite[Proposition~3.7]{Putnam:Binary}, $\sim_{\xi}$ is an equivalence relation and $\mu \sim_{\xi} \nu$ for $\mu,\nu \in E^{-\infty}$ implies $\sigma(\mu)\sim_{\xi}\sigma(\nu)$. Therefore, if we denote by $\mathcal{J}_{\xi}$ the quotient space $E^{-\infty}/ \sim_{\xi}$, the shift $\sigma$ descends to a continuous mapping $\sigma_{\xi}:\mathcal{J}_{\xi} \to \mathcal{J}_{\xi}$. 

Putnam shows in \cite[Section~3]{Putnam:Binary} that if $E$ is assumed \textit{primitive} and $\xi$ satisfies an extra hypothesis $(H3)$, then $\sigma_{\xi}$ is an expanding surjective local homeomorphism. In addition, he describes the expanding metric in great detail. We will show the same, but with hypothesis $(H3)$ removed and primitive weakened to \textit{no sources} by proving that $(\sigma_{\xi}, \mathcal{J}_{\xi})$ is isomorphic to the limit space dynamical system of a contracting and regular self-similar groupoid acting on $E$. We then apply the recent work on these dynamical systems in \cite{BBGHSW}. We do not, however, extend Putnam's metric results.

\subsection{Putnam's binary factor maps as self-similar groupoid actions on graphs}\label{Putnam systems}
In this section we show that Putnam's construction gives rise to a self-similar groupoid on a graph.

Let $\xi = \xi^{0},\xi^{1}:H \to E$ be an embedding pair. Denote $H^{0}_{\xi}\coloneqq \xi^{0}(H^{0}) = \xi^{1}(H^{0})$ and $H_{\xi}^{1} \coloneqq \xi^{0}(H^{1})\cup \xi^{1}(H^{1})$. Let $\tilde{G}_{\xi} = \mathbb{Z}\times E^{0}$, with groupoid structure determined by the projection $\pi:\tilde{G}_{\xi} \to E^{0}$. This means that $(m,v),(n,w)\in \tilde{G}_{\xi}$ are composable if and only if $v = w$, in which case $(m,v)(n,v) = (m+n, v)$. Hence, $d(m,v)=c(m,v)= \pi(m,v)$, so that $\tilde{G}_{\xi}$ is a \textit{group bundle} in the sense that $\pi^{-1}(v)=\ZZ$ for all $v \in E^0$.

For $v\in H^{0}$, let $\ell(v)$ be the maximum length of a path $y$ in $H$ satisfying $r(y) = v$. Consider the quotient bundle $G_{\xi} = (\cup_{v\in H^{0}} \mathbb{Z}/2^{\ell(v)}\mathbb{Z}\times\{\xi^{0}(v)\}) \cup (\{0\}\times (E^{0}\setminus H_{\xi}^{0}))$, where we make the convention that if $\ell(v) = \infty$, then $\mathbb{Z}/2^{\ell(v)}\mathbb{Z} = \mathbb{Z}$. We aim to define a (faithful) self-similar groupoid action of $G_{\xi}$ on $E$.

For $(m,v) \in \tilde{G}_{\xi}=\mathbb{Z}\times E^{0}$ and $e\in vE^{1}$, we define 
\begin{align}\label{Putnam_action}
(m,v) \cdot e &=
\begin{cases}
e \text{ if } e \notin H^{1}_{\xi},\\
\xi^{j}(h)\text{ if } e =\xi^{i}(h)\text{ for } h\in H^{1}, i,j,n \in\{0,1\}\text{ such that } m+i = 2n+j 
\end{cases}\\ \notag
(m,v)|_{e} &= 
\begin{cases}
(0,s(e)) \text{ if } e \notin H^{1}_{\xi},\\
(n, s(e))\text{ if } e =\xi^{i}(h)\text{ for }h\in H^{1}, i,j,n\in\{0,1\}\text{ such that } m+i = 2n+j. 
\end{cases}
\end{align}
This defines an action-restriction pair in the sense of Definition \ref{dfn:action_restriction}.  To find the kernel of the induced action $\phi:\tilde{G}_{\xi}\to \PI(E^{*})$, let us describe the action and restriction in terms of the binary odometer action. Let $\alpha:\mathbb{Z} \curvearrowright \{0,1\}^{*}$ be the self-similar group representation of $\mathbb{Z}$ by the  $2$-odometer action \cite[Section~1.7.1]{Nekrashevych:Self-similar}.  In our language, this is the self-similar group defined by the Katsura pair $A = (2)$ and $B = (1)$. More explicitly, for $m\in\mathbb{Z}$ and $i\in\{0,1\}$, we have
\[
a^{m}\cdot i = j\text{ and }a^{m}|_{i} = a^{n},\text{ where }m + i = 2n + j.
\]
For example, we have $a\cdot 1^{k} = 0^{k}$, $a|_{1^{k}} = a$ and $a\cdot 0 = 1$, $a|_{0} = \text{id}$. This completely describes $a$ as an automorphism of $\{0,1\}^{*}$.

For $n\geq 0$, $i = i_{1}i_{2}..i_{n}\in\{0,1\}^{n}$, and $h = h_{1}h_{2}...h_{n}\in H^{n}$, let 
\begin{equation}\label{odometer_embedding}
\xi^{i}(h) \coloneqq \xi^{i_{1}}(h_{1})\xi^{i_{2}}(h_{2})...\xi^{i_{n}}(h_{n}).
\end{equation}
Notice that the embedding $\xi:H^1 \to E^1$ extends to an embedding $\xi^i:H^n \to E^n$ via \eqref{odometer_embedding}. If $\mu \in E^n$ satisfies $\mu = \xi^{i}(h)$ for some $i\in \{0,1\}^{n}$ and $h\in vH^{n}$, then for any $\nu \in s(\mu)E^*$ we have
\begin{equation}
\label{eq:Putnam_odometer}
(m,r(\mu)) \cdot \mu\nu = \xi^{(a^m \cdot i)}(h)\, (a^m|_i, s(\mu)) \cdot \nu. 
\end{equation}
If $e\notin H_{\xi}^{1}$, then for any $\nu\in s(e)H^{*}$, we have $(m,r(e)) \cdot e\nu = e\nu$. Thus, the action $\phi:\tilde{G}_{\xi}\to \PI(E^{*})$ is completely described by the $2$-odometer action and the trivial action.

Since the kernel of $\alpha:\mathbb{Z}\curvearrowright\{0,1\}^{n}$ is $2^{n}\mathbb{Z}$, the kernel of the $\tilde{G}_{\xi}$-action on $E^{*}$ is given by $(\cup_{v \in H^{0}}2^{\ell(v)}\mathbb{Z}\times\{\xi^{0}(v)\}) \cup (\mathbb{Z} \times (E^{0}\setminus H_{\xi}^{0}))$. Hence, the quotient of the $\tilde{G}_{\xi}$-action is $G_{\xi}$. The quotient action $\phi:G_{\xi}\to \PI(E^*)$ is then self-similar by the results in Section \ref{subsec:ssa}.

\begin{center}
\begin{figure}
\begin{tikzpicture}[scale=1]
\node at (-2,0) {$H:$};
\node[vertex] (vertexy) at (-1,0) {$v$}
	edge [->,>=Straight Barb,out=60,in=120,loop] node[above]{$e$} (vertexy);
\draw [-to](0,0) -- (1,0) node[midway, above]{$\xi$};
\node[vertex] (vertexyy) at (2,0) {$v$}
	edge [->,>=Straight Barb,out=-30,in=30,loop] node[right]{$f$} (vertexyy)
	edge [->,>=Straight Barb,out=240,in=300,loop] node[below]{$e_1$} (vertexyy)
	edge [->,>=Straight Barb,out=60,in=120,loop] node[above]{$e_0$} (vertexyy);
\node at (4,0) {$:E$};
\end{tikzpicture}
\caption{The embedding pair $\xi: H \to E$ for Example \ref{Putnam emb ex}.}
\label{Putnam groupoid ex}
\end{figure}
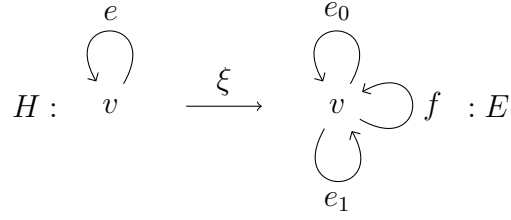
\end{center}

\begin{exm}\label{Putnam emb ex}
Consider the graphs $E$ and $H$ in Figure \ref{Putnam groupoid ex} along with the embedding pair $\xi :H \to E$ defined by 
$\xi^{0}(e)=e_{0}, \, \xi^{1}(e)=e_{1}.$
Then \eqref{Putnam_action} gives partial isomorphisms generating a self-similar groupoid $(G_\xi,E)$ via: 
\begin{align*}
&(1,v) \cdot e_{0} =e_{1} \quad (1,v)|_{e_{0}}=(0,v), & &(1,v) \cdot e_{1} =e_{0} \quad (1,v)|_{e_{1}}=(1,v),\\
&(1,v) \cdot f =f \quad (1,v)|_f=(0,v)
\end{align*}
$(G_{\xi}, E)$ is not a self-similar groupoid action arising from a KEP-action. Indeed, if $(G_{\xi}, E)$ was isomorphic to $(G_{B}, E_{A})$ for some Katsura pair, then $A = (3)$ and $B = (n)$ for some $n\in\mathbb{Z}$. The action of $G_{\xi} = \mathbb{Z}$ on $E$ is non-trivial and contracting, so its contracting coefficient satisfies $0 <\rho = \frac{|n|}{3} < 1$. So either $|n| = 1$ or $|n| = 2$, but neither of these cases allow for $\mathbb{Z}\cdot f = f$, $\mathbb{Z}|_{f} = 0$. Thus, $(G_{\xi}, E)$ is not isomorphic to a KEP-action.

However, we will show in Section \ref{out-splits} that every self-similar groupoid from an embedding pair is an \emph{out-splitting} of a KEP-action. \qed
\end{exm}

\subsection{Properties of $(G_{\xi}, E)$}

In this section we prove that the self-similar groupoid actions associated with a binary factor are contracting and regular. Moreover, we prove that the shift map on the limit space of the self-similar groupoid is conjugate to Putnam's expanding local homeomorphism on the quotient space $\mathcal{J}_{\xi}=E^{-\infty}/\sim_{\xi}$ described in Section \ref{Putnam's construction}. 

\begin{prp}
Let $\xi = \xi^{0},\xi^{1}:H \to E$ be an embedding pair. Then, $(G_{\xi}, E)$ is contracting and regular.
\end{prp}
\begin{proof}
We first show $(G_{\xi},E)$ is contracting. We show the nucleus is contained in $\mathcal{N} = (\{-1,0,1\}\times H_{\xi}^{0})\cup \{0\}\times (E^{0}\setminus H_{\xi}^{0}).$
    Let $g = (m,w)\in G_{\xi}$. If $w\notin H^{0}_{\xi}$, then $m = 0$ and hence $g\in\mathcal{N}$. So suppose $w = \xi^{0}(v)$. First, assume $\ell(v) < \infty$. Then, if there is a path $\mu \in wE^{*}$ such that $|\mu| = n\geq \ell(v) +1$, at least one of its edges $\mu_{k}\notin H_{\xi}^{1}$. So, we have by \eqref{Putnam_action} that 
\begin{equation}\label{notinimageofH}
 g|_{\mu} = (g|_{\mu_{1}...\mu_{k}})|_{\mu_{k+1}..\mu_{n}} = (0,s(\mu_{k}))|_{{\mu_{k+1}..\mu_{n}}} = (0, s(\mu_{n})) \in \mathcal{N}.
 \end{equation}
If there are no paths $\mu$ of length $|\mu|\geq \ell(v)+1$ in $E^{*}$ ending at $\xi^{0}(v)$, then the contracting condition is satisfied for $g$ vacuously.

Now, suppose $w =\xi^{0}(v)$ and $\ell(v) = \infty$. The 2-odometer action of $\mathbb{Z}$ is contracting, with $\mathcal{N}=\{-1,0,1\}$, see \cite[Section~1.7.1]{Nekrashevych:Self-similar}. So, let $K\in\mathbb{N}$ be the number such that $a^m|_{i}\in\{-1,0,1\}$ for all $|i|\geq K$. Let $\mu \in wE^{*}$ have length $|\mu| = n\geq K$. If there is $k\leq n$ such that $\mu_{k}\notin \xi^{0}(H^{1})\cup \xi^{1}(H^{1})$, then \eqref{notinimageofH} implies $g|_{\mu} = (0, s(\mu)) \in \mathcal{N}$. Otherwise, $\mu_{1}...\mu_{n} = \xi^{i}(h)$ for some $i\in\{0,1\}^{n}$ and $h\in H^{n}$. By \eqref{eq:Putnam_odometer}, letting $a^{l} = a^{m}|_{i}$, we have 
\[
g|_{\mu}  = (l, s(\mu))\in\mathcal{N}.
\]
Thus, $(G_{\xi}, E)$ is contracting.

We now show $(G_{\xi}, E)$ is regular. Let $g = (m,w)$ and $\mu\in wE^{*}$ satisfy $g\cdot \mu = \mu$ and $g|_{\mu}\neq s(\mu)$. Then,  \eqref{notinimageofH} implies $\mu = \xi^{i}(h)$ for some $i\in \{0,1\}^{*}$ and $h\in H^{*}$. By \ref{eq:Putnam_odometer}, we have
\[\xi^{a^{m}\cdot i}(h) = g\cdot \mu = \mu = \xi^{i}(h).\]
The $2$-odometer action is regular, so let $M\in\mathbb{N}$ be such that if $a^{m}\cdot \tilde{i} = \tilde{i}$ and $|\tilde{i}| > M$, then $a^{m}|_{\tilde{i}} = e$. Hence, we have $|\mu| = |i| \leq M$. Therefore, $(G_{\xi}, E)$ is regular.
\end{proof}

\subsection{Equality of the dynamical systems $(\sigma_{\xi}, \mathcal{J}_{\xi})$ and $(\tilde{\sigma}, \Jj_{G_{\xi},E})$}

In this section we prove the following.

\begin{thm}
\label{thm:equality}
Let $\xi = \xi^{0},\xi^{1}:H \to E$ be an embedding pair and $(G_{\xi},E)$ be the associated self-similar groupoid action on $E$. Then, the equivalence relation $\sim_{\xi}$ on $E^{-\infty}$ is equal to the asymptotic equivalence relation $\sim_{ae}$ on $E^{-\infty}$. Thus, $(\sigma_{\xi}, \mathcal{J}_{\xi}) = (\tilde{\sigma}, \Jj_{G_{\xi},E})$.
\end{thm}

We will begin by first proving a couple of lemmas.
\begin{lem} \label{lem:equality1}
Let $\xi = \xi^{0},\xi^{1}:H \to E$ be an embedding pair and $(G_{\xi},E)$ be the associated self-similar groupoid action on $E$. Suppose $\mu \in E^{-\infty}$ has the property that $\mu_{j}\notin H_{\xi}^{1}$ for infinitely many $j\in\mathbb{N}$. Then, $\mu$ is only asymptotically equivalent to itself.
\end{lem}
\begin{proof}
Suppose $\nu \in E^{-\infty}$ and $\mu \sim_{ae} \nu$. Let $F\subseteq G_{\xi}$ be a finite set and $(g_{n})_{n<0}\subseteq F$ be a sequence such that $d(g_{n}) = r(\mu_{n})$ and $g_{n}\cdot \mu_{n}...\mu_{-1} = \nu_{n}...\nu_{-1}$ for all $n<0$. Let $(j_{n})_{n<0}\subseteq \{n<0 \mid n \in \mathbb{Z}\}$ be a decreasing sequence such that $\mu_{j_{n}}\notin H_{\xi}^{1}$ for all $n<0$. Then, by definition of the action, we have that $g_{j_{n}}\cdot \mu_{j_{n}} = \mu_{j_{n}}$, $g_{j_{n}}|_{\mu_{j_{n}}} = (0, s(x_{j_{n}}))$ and hence $\nu_{j_{n}}...\nu_{-1} = g_{j_{n}}\cdot \mu_{j_{n}}...\mu_{-1} = \mu_{j_{n}}...\mu_{-1}$ for all $n<0$. So $\mu = \nu$.
\end{proof}

\begin{lem}
    \label{lem:equality2}
Let $\xi = \xi^{0},\xi^{1}:H\to E$ be an embedding pair and $(G_{\xi},E)$ be the associated self-similar groupoid action on $E$.  For $\mu,\nu \in E^{-\infty}$, $\mu \sim_{ae} \nu$ if and only if $\mu = \nu$ or there is $n<0$ such that $(y_{k})_{k\leq n} \subseteq H^{1}$, $(i_{k})_{k\leq n}, (i'_{k})_{k\leq n}\subseteq \{0,1\}$ with $\mu_{k} = \xi^{i_{k}}(y_{k})$, $\nu_{k} = \xi^{i'_{k}}(y_{k})$, for all $k\leq n$, $(...i_{n-1}i_{n})\sim_{ae}(...i'_{n-1}i'_{n})$ relative to the 2-odometer action of $\mathbb{Z}$, and one of the following hold:
    \begin{enumerate}
        \item[(1)*] $n = -1$, or
        \item[(2)*] $n < 1$, $\mu_{j} = \nu_{j}$ for all $j\geq n+1$ and $\mu_{n+1} = \nu_{n+1}\notin H_{\xi}^{1}$.
\end{enumerate}
\end{lem}

\begin{proof}
We prove the forward direction first. Suppose that $\mu\sim_{ae} \nu$ and $\mu\neq \nu$. From Lemma \ref{lem:equality1}, we know that there is $n<0$ such that for all $k \leq n$, there is $i_{k}, i'_{k} \in \{0,1\}$ and $y_{k}, y'_{k}\in H^{1}$ with $\mu_{k} = \xi^{i_{k}}(y_{k})$ and $\nu_{k} = \xi^{i'_{k}}(y'_{k})$. Let $n$ be the largest such $n$ for which this is true, and let $(g_{k})_{k<0}$ be a sequence contained in some finite set $F$ of $G_{\xi}$ satisfying $d(g_{k}) = r(\mu_{k})$ and $g_{k}\cdot \mu_{k}...\mu_{-1} = \nu_{k}...\mu_{-1}$ for all $k<0$. 

Let $\tilde{F}$ be a finite set in $\mathbb{Z}$ such that $(g_{k})_{k\leq n}\subseteq \pi(\tilde{F}\times H_{\xi}^{0})$, where $\pi:\tilde{G}_{\xi}\to G_{\xi}$ is the quotient map. Then, there is $(n_{k})_{k\leq n}\subseteq \tilde{F}$ such that 
\[
\xi^{a^{n_{k}}\cdot (i_{k}..i_{n})}(y_{k}..y_{n}) = g_{k}\cdot \mu_{k}...\mu_{n} = \nu_{k}...\nu_{n} = \xi^{i'_{k}...i'_{n}}(y'_{k}..y'_{n}).
\]
It follows that $y_{k} = y'_{k}$ for all $k\leq n$ and $(...i_{n-1}i_{n})\sim_{ae} (...i'_{n-1}i'_{n})$.

Now, if $n =-1$, then we are done, so assume $n<1$ and, without loss of generality, that $\mu_{n+1}\notin H^{1}_{\xi}$. Then $\nu_{n+1} = g_{n+1}\cdot \mu_{n+1} = \mu_{n+1}$ and $g_{n+1}|_{\mu_{n+1}} = (0, s(\mu_{n+1}))$, so that  $\nu_{n+1}...\nu_{-1} = g_{n+1}\cdot \mu_{n+1}..\mu_{-1} = \mu_{n+1}...\mu_{-1}$.

Now, we prove the reverse direction. Let $\tilde{F}$ be a finite set in $\mathbb{Z}$ and $(n_{k})_{k\leq n}\subseteq \tilde{F}$ a sequence such that $\alpha^{n_{k}}\cdot i_{k}...i_{n} = i'_{k}...i'_{n}$ for all $k\leq n$. It follows from the definition of the action of $G_{\xi}$ that if we let $g_{k} = \pi((n_{k}),s(\mu_{k}))$, then $g_{k}\cdot \mu_{k}...\mu_{n} = \nu_{k}...\nu_{n}$ for all $k\leq n$. If $n=-1$, then we are done, so assume $n < -1$. Since $\mu_{n+1} = \nu_{n+1}\notin H^{1}_{\xi}$, it and its extension to the path $\mu_{n+1}...\mu_{-1} = \nu_{n+1}...\nu_{-1}$ are fixed by any element in $G_{\xi}\cap d^{-1}(s(\mu_{n+1}))$. Then, 
\[
g_{k}\cdot \mu_{k}...\mu_{-1} = \nu_{k}...\nu_{n}(g|_{\mu_{n}}\cdot \mu_{n+1}..\mu_{-1}) = \nu_{k}...\nu_{n} \nu_{n+1}((0,s(\mu_{n+1}) \cdot \mu_{n+2}..\mu_{-1}) = \nu_{k}...\nu_{-1}.
\]
So, if we let $g_{j} = (0, s(\mu_{j}))$ for $j\geq n+1$, then $(g_{k})_{k<0}$ is a sequence contained in $\pi(\tilde{F}\times H^{0}_{\xi})\cup \{0\}\times (E^{0}\setminus H^{0}_{\xi})$ implementing $\mu \sim_{ae} \nu$.
\end{proof}

We will now prove the main theorem.
\begin{proof}[Proof of Theorem~\ref{thm:equality}]
Recall from \cite[Section~3.1.2]{Nekrashevych:Self-similar} that two sequences $(...i_{n-1}i_{n})$ and $(...i'_{n-1}i'_{n})$ are asymptotically equivalent relative to the 2-odometer action of $\mathbb{Z}$ if and only if either $i_{k} = i'_{k}$ for all $k\leq n$, or there is $n'\leq n$ and $i\in\{0,1\}$ such that $i_{k} = i$ and $i'_{k} = 1-i$ for all $k\leq n'$, and one of the following hold:
\begin{enumerate}
\item[(1)**] $n = n'$, or
\item[(2)**] $n' < n$, $x_{k} = x'_{k}$ for all $n
\geq k > n'+1$, $i_{n'+1} = 1-i$ and $i'_{n'+1} = i$.
\end{enumerate}
Combining this description of asymptotic equivalence for the 2-odometer action with the description of asymptotic equivalence for $(G_{\xi}, E)$ in Lemma \ref{lem:equality2} yields equality with $\sim_{\xi}$. For clarity, cases $(1)$, $(2)$ and $(3)$ of $\sim_{\xi}$ correspond, respectively, to cases $(1)^{*} + (1)^{**}$ ($n'= n = -1$), $(2)^{**}$ ($n' < n \leq -1$) , and $(2)^{*} + (1)^{**}$ ($n' = n < -1$) of Lemma \ref{lem:equality2} and above.
\end{proof}

\section{Out-splittings and KEP-action models for embedding pairs}\label{out-splits}

In this section, we determine the relationship of $(G_{\xi}, E)$ with a KEP-action model. In particular, we show there are matrices $A\in M_{N}(\{0,1,2\})$ and $B\in M_{N}(\{0,1\})$ such that $(\tilde{\sigma}, \Jj_{G_{\xi}, E})$ is topologically conjugate to $(\tilde{\sigma},\Jj_{G_B,E_A})$. We will do so by showing \textit{out-splits} of self-similar group bundles preserve limit spaces, and that a certain KEP-action arises as an out-split of $(G_{\xi},E)$.

\subsection{Out-splits of self-similar group bundles}
We take the approach on out-splits found in \cite{BMR}. For the ``classical'' approach, see \cite{Kitchens}.  Let $E$ be a directed graph and let  $OS = (\pi, \beta)$ be a tuple, where $\pi: E^{1}\to E^{0}_{OS}$ and $\beta:E_{OS}^{0}\to E^{0}$ are maps such that $s = \beta\circ \pi$. The \textit{out-split of $E$ by $OS$} is the graph $E_{OS} = (E^{0}_{OS}, E^{1}_{OS}, r_{OS}, s_{OS})$, where $E_{OS}^{1} \coloneqq E_{OS}^{0}\tensor[_\beta]{\times}{_r}E^{1}$ and $r_{OS},s_{OS}:E^{1}_{OS}\to E_{OS}^{0}$ are defined for $(v,e)\in E_{OS}^{1}$ as $r_{OS}(v,e) = v$ and $s_{OS}(v,e) = \pi(e)$.

\begin{center}
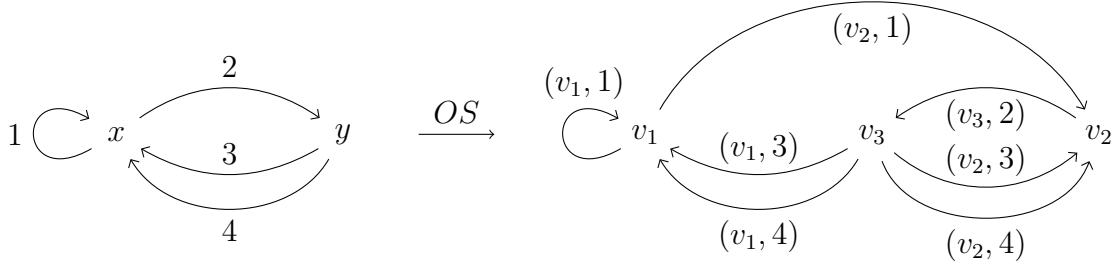
\begin{figure}
\begin{tikzpicture}[scale=1]
\node at (-1,0) {$y$};
\node at (-4,0) {$x$};
\node[vertex] (vertexe) at (-1,0)   {$\quad$};
\node[vertex] (vertex-a) at (-4,0)   {$\quad$}
	edge [->,>=Straight Barb,out=35,in=145] node[above,swap,pos=0.5]{$2$} (vertexe)
	edge [->,>=Straight Barb,out=210,in=150,loop] node[left,pos=0.5]{$1$} (vertex-a)
	edge [<-,>=Straight Barb,out=330,in=210] node[above,swap,pos=0.5]{$3$} (vertexe)
	edge [<-,>=Straight Barb,out=300,in=240] node[below,swap,pos=0.5]{$4$} (vertexe);
\draw [-to](0,0) -- (1,0) node[midway, above]{$OS$};
\node at (3,0) {$v_1$};
\node at (6,0) {$v_3$};
\node at (9,0) {$v_2$};
\node[vertex] (vertexer) at (6,0)   {$\quad$};
\node[vertex] (vertex-ar) at (3,0)   {$\quad$}
	edge [->,>=Straight Barb,out=210,in=150,loop] node[above,pos=0.8]{$(v_1,1)$} (vertex-ar)
	edge [<-,>=Straight Barb,out=330,in=210] node[above,swap,pos=0.5]{$(v_1,3)$} (vertexer)
	edge [<-,>=Straight Barb,out=300,in=240] node[below,swap,pos=0.5]{$(v_1,4)$} (vertexer);
\node[vertex] (vertex-br) at (9,0)   {$\quad$}
	edge [->,>=Straight Barb,out=145,in=35] node[below,swap,pos=0.5]{$(v_3,2)$} (vertexer)
	edge [<-,>=Straight Barb,out=120,in=60] node[below,pos=0.5]{$(v_2,1)$} (vertex-ar)
	edge [<-,>=Straight Barb,out=220,in=320] node[above,swap,pos=0.5]{$(v_2,3)$} (vertexer)
	edge [<-,>=Straight Barb,out=250,in=290] node[below,swap,pos=0.5]{$(v_2,4)$} (vertexer);
\end{tikzpicture}
\caption{An out-split of the graph on the left appears on the right.}
\label{out-split ex}
\end{figure}
\end{center}

\begin{exm}\label{example3}
Let $E$ be the graph on the left of Figure \ref{out-split ex}. Let $E_{OS}^0 \coloneqq \{v_1,v_2,v_3\}$ and $\pi: E^{1}\to E^{0}_{OS}$ be defined by 
\[
\pi(1)=v_1,\,\, \pi(2)=v_2,\,\, \pi(3)=v_3, \text{ and } \pi(4)=v_3.
\]
Since $s = \beta\circ \pi$, the map $\beta:E_{OS}^{0}\to E^{0}$ is given by
\[
\beta(v_1)= \beta(\pi(1))= s(1)=x , \,\, \beta(v_2)=x, \text{ and } \beta(v_3)=y. 
\]
Then,
\[
E_{OS}^{1} \coloneqq E_{OS}^{0}\tensor[_\beta]{\times}{_r}E^{1}=\{(v_1,1), (v_1,3), (v_1,4), (v_2,1), (v_2,3), (v_2,4), (v_3,2)\},
\]
with the out-split graph $E_{OS}$ depicted on the right of Figure \ref{out-split ex}. \qed
\end{exm}
It's routine to check that the dynamical systems $(\sigma, E^{-\infty})$ and $(\sigma, E_{OS}^{-\infty})$ are topologically conjugate via the map $I: E^{-\infty} \to E_{OS}^{-\infty}$ given by 
\[
I(...,e_{-2}, e_{-1})= (...,(\pi(e_{-3}), e_{-2}), (\pi(e_{-2}), e_{-1})).
\]
More generally, for every $n\in\mathbb{N}$, there is a bijection  $I_{n}: E_{OS}^{0}\tensor[_\beta]{\times}{_r}E^{n}\to E_{OS}^{n}$ defined for $(v, \mu)\in E_{OS}^{0}\tensor[_\beta]{\times}{_r}E^{n}$ with $\mu =e_{-n} \ldots e_{-1}$ by 
\[
I_{n}(v,\mu) = (v,e_{-n})(\pi(e_{-n}), e_{-n+1}) \ldots (\pi(e_{-2}), e_{-1}).
\]

Now suppose $(G,E)$ is a self-similar group bundle; that is, the action of $G$ on $E^*$ is range preserving. We can define a new group bundle $(G_{OS},E_{OS})$ with 
\[
G_{OS} = \{(g,v)\in G \times E_{OS}^{0} \mid d(g)=c(g)=\beta(v)\},
\]
where $(g,v)$ and $(g, v')$ are composable if and only if $v = v'$, in which case 
\[
(g,v)\cdot (g',v) = (gg', v).
\]
The action-restriction pair $(G_{OS},E_{OS})$ is defined for $(g,v)\in G_{OS}$ and $(v, e)\in E_{OS}^{1}$, as
\begin{equation}\label{out-split group}
(g,v) \cdot (v, e) = (v,g\cdot e) \quad \text{ and } \quad (g, v)|_{(v, e)} = (\pi(e), g|_{e}).
\end{equation}
The induced action of $(g,v)$ on a path $I_{n}(v,e)$, for $(v,e)\in E_{OS}^{0}\tensor[_\beta]{\times}{_r}E^{n}$, is 
\begin{equation}\label{out-split action}
(g,v)\cdot I_{n}(v,e) = I_{n}(v, g\cdot e),
\end{equation}
making it clear that the homomorphism $\phi:G_{OS}\to\PI(E^{*}_{OS})$ is faithful. Therefore, $(G_{OS}, E_{OS})$ is a self-similar groupoid action on a graph. We call $(G_{OS}, E_{OS})$ the \textit{out-split of $(G,E)$ by $OS$}.

\begin{thm}
    \label{thm:out-split}
    Let $(G,E)$ be a self-similar group bundle and $E_{OS}$ an out-split of $E$. Then, for $\mu,\nu\in E^{-\infty}$, $\mu$ is asymptotically equivalent to $\nu$ relative to $(G,E)$ if and only if $I(\mu)$ is asymptotically equivalent to $I(\nu)$ relative to the out-split $(G_{OS}, E_{OS})$. Consequently, $(\tilde{\sigma}, \Jj_{G,E})$ is topologically conjugate to $(\tilde{\sigma}, \Jj_{G_{OS},E_{OS}})$. 
\end{thm}

\begin{proof}
If $F\subseteq G$ is a finite set, we let $F_{OS} = \{(g,v)\in F\times E_{OS}^{0}: d(g) = \beta(v)\}$. Then, $\mu$ is asymptotically equivalent to $\nu$ if and only if there is a sequence $(g_{n})_{n<0}$ contained in some finite set $F$ of $G$, such that $d(g_{n}) = r(\mu_{n})$ and $g_{n}\cdot \mu_{n}....\mu_{-1} = \nu_{n}...\nu_{1}$ for all $n<0$, if and only if there is a sequence $(g_{n})_{n<0}\subseteq G$ and a finite set $F\subseteq G$ such that $((g_{n},\pi(\mu_{n-1})))_{n<0}$ is contained in $F_{OS}$ and satisfies $(g_{n},\pi(\mu_{n-1}))\cdot I_{-n}(\pi(\mu_{n-1}), \mu_{n}...\mu_{-1}) = I_{-n}(\pi(\nu_{n-1}),\nu_{n}...\nu_{-1})$ for all $n<0$, if and only if $I(\mu)$ is asymptotically equivalent to $I(\nu)$.
\end{proof}
\begin{rmk}
    A number of properties are preserved by out-splitting self-similar group bundles. For instance, using equations \eqref{out-split group} and \eqref{out-split action}, it is easy to see that $(G, E)$ is contracting (or regular) if and only if $(G_{OS}, E_{OS})$ is contracting (or regular).
\end{rmk}

\subsection{KEP-action models for embedding pairs}\label{Katsura models Putnam}

Suppose $\xi = \xi^{0},\xi^{1}:H\to E$ is an embedding pair and $(G_{\xi}, E)$ its associated self-similar groupoid action. We show there is a Katsura pair $(A,B)$ such that $(G_{A},E_{B})$ is the out-split of $(G_{\xi}, E)$.

Let $E_{OS}^{0}$ be the set obtained from $E^{1}$ by identifying the edges $\xi^{0}(h)$ and $\xi^{1}(h)$, for all $h\in H^{1}$. Let $\pi:E^{1}\to E_{OS}^{0}$ be the quotient map. Since the edges being identified share the same source, there is a unique map $\beta:E_{OS}^{0}\to E^0$ satisfying $\beta \circ \pi = s$. Therefore, the tuple $OS = (\pi, \beta)$ determines an out-split $(G_{OS}, E_{OS})$ of the self-similar group bundle $(G_{\xi}, E)$. We show $(G_{OS},E_{OS})$ is isomorphic to a KEP-action. 

If $w\in \pi(H_{\xi}^{1})$, there is a unique $h\in H^{1}$ such that $\pi^{-1}(w) = \{\xi^{0}(h),\xi^{1}(h)\}$. Otherwise,
$\pi^{-1}(w) = \{w\}\subseteq E^{1}\setminus H_{\xi}^{1}$.

Therefore, for $v,w\in E_{OS}^0$,  if $A_{v,w}\coloneqq |v(E_{OS}^{1})w|> 0$, we have 
\begin{align}
v(E_{OS}^{1})w &=
\begin{cases}
\{(v, \xi^{0}(h)), (v,\xi^{1}(h))\} \text{ for some }h\in H^{1} \text{ if } w\in\pi(H_{\xi}^{1}),\\\\
\{(v, w)\} \text{ otherwise.} 
\end{cases}
\end{align}
In the first case, we shall denote $e_{v,w, m}\coloneqq (v, \xi^{m}(h))$ for $m\in\{0,1\}$ and in the second case, denote $e_{v,w,0} \coloneqq (v,w)$. Replacing the notation $(k,v)\in G_{OS}$ with $a_{v}^{k}$, we see that when $A_{v,w}\neq 0$, the action and restriction  $(G_{OS})_{v}\times v(E_{OS}^{1})w\to v(E_{OS}^{1})w\times (G_{OS})_{w}$ is given by $(a_{v}^{k}, e_{v,w,m})\to (e_{v,w, \hat{m}}, a_{w}^{\hat{k}})$, where
\begin{equation}\label{Putnam_out-split}
k(A_{v,w}-1) + m = \hat{k}(A_{v,w}) + \hat{m}, \text{ and } 0\leq m,\hat{m} < A_{v,w} - 1.
\end{equation}
Comparing \eqref{Putnam_out-split}  with \eqref{ExelPardoKatsura_action} and \eqref{eq:Katsura_action-restriction} ,  we see that $(G_{OS}, E_{OS})$ is canonically isomorphic to $(G_{B}, E_{A})$, where $B =  (\max\{0, A_{v,w} - 1\})_{v,w\in E_{OS}^{0}}$. We record this formally as a corollary to Theorem \ref{thm:out-split}.

\begin{cor}
Let $\xi = \xi^{0},\xi^{1}:H\to E$ be an embedding pair and let $(G_{\xi},E)$ be its associated self-similar groupoid action, as described in Section \ref{Putnam systems}. Then, the out-split $(G_{OS},E_{OS})$ of $(G_{\xi}, E)$, described in Section \ref{Katsura models Putnam} is canonically isomorphic to the KEP-action $(G_{B}, E_{A})$ with $A = (A_{v,w})_{v,w \in E_{OS}^{0}}$ the adjacency matrix of $E_{OS}$ and $B =  (\max\{0, A_{v,w} - 1\})_{v,w\in E_{OS}^{0}}$. Moreover, the limit spaces $(\tilde{\sigma},\Jj_{G_{\xi}, E})$ and $(\tilde{\sigma}, \Jj_{G_B,E_A})$ are topologically conjugate.
\end{cor}

\begin{center}
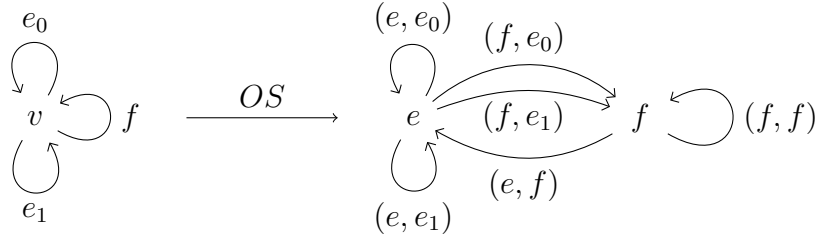
\begin{figure}
\begin{tikzpicture}[scale=1]
\node[vertex] (vertexyy) at (-2,0) {$v$}
	edge [->,>=Straight Barb,out=-30,in=30,loop] node[right]{$f$} (vertexyy)
	edge [->,>=Straight Barb,out=240,in=300,loop] node[below]{$e_1$} (vertexyy)
	edge [->,>=Straight Barb,out=60,in=120,loop] node[above]{$e_0$} (vertexyy);
\draw [-to](0,0) -- (2,0) node[midway, above]{$OS$};
\node[vertex] (vertexer) at (6,0)   {$f$}
	edge [->,>=Straight Barb,out=-30,in=30,loop] node[right]{$(f,f)$} (vertexer);
\node[vertex] (vertex-ar) at (3,0)   {$e$}
	edge [->,>=Straight Barb,out=240,in=300,loop] node[below]{$(e,e_1)$} (vertex-ar)
	edge [->,>=Straight Barb,out=60,in=120,loop] node[above]{$(e,e_0)$} (vertex-ar)
	edge [->,>=Straight Barb,out=20,in=160] node[below,swap,pos=0.5]{$(f,e_1)$} (vertexer)
	edge [->,>=Straight Barb,out=40,in=140] node[above,swap,pos=0.5]{$(f,e_0)$} (vertexer)
	edge [<-,>=Straight Barb,out=330,in=210] node[below,swap,pos=0.5]{$(e,f)$} (vertexer);
\end{tikzpicture}
\caption{The out-split associated with Examples \ref{Putnam emb ex} and \ref{out-split emb ex}.}
\label{out-split groupoid ex}
\end{figure}
\end{center}

\begin{exm}\label{out-split emb ex}
Consider again Example \ref{Putnam emb ex}. We have $E^{0}_{OS} = \{e,f\}$ and the quotient map $\pi:E^{1}\to E^{0}_{OS}$ satisfies 
\[
\pi(e_{0})=\pi(e_{1}) = e, \, \pi(f)=f.
\]
and 
\[
E_{OS}^1=\{(f,f), (e,f), (e,e_{i}),(f,e_{j}) \mid i,j\in\{0,1\}\}.
\]
The graph $E$ and its out-split $OS$ is recorded in Figure \ref{out-split groupoid ex}. If we order $e< f$, then 
the KEP-action is defined by the matrices
\[
A = \left( \begin{matrix}
2 & 1 \\
2 & 1 \\
\end{matrix} \right) \qquad \text{ and } \qquad 
B = \left( \begin{matrix}
1 & 0\\
1 & 0 \\
\end{matrix} \right). \qed
\]
\end{exm}

\section{KEP-systems as bundles of odometers}\label{Bundles of odometers}
We now provide a description of the KEP-systems $(\tilde{\sigma},\Jj_{G_B,E_A})$ as a bundle of dynamical systems that fibre over the shift space of the connectivity graph of $E_{A}$ when $B$ is a matrix taking values in $\{0,1\}$. Note that from the previous section, this class includes the dynamical systems arising from embedding pairs.

We will go even farther and first describe the connected component space of the limit space for an arbitrary finitely generated and contracting self-similar groupoid $(G, E)$. This result does not appear in the literature anywhere else.

Recall that for a topological space $X$, its \textit{connected component space} $\mathcal{C}(X)$ is the quotient of $X$ by the equivalence relation $\sim_{C}$, where $x\sim_{C}y$ if and only if $x$ and $y$ are in the same connected component.
\par
For a self-similar groupoid $(G, E)$ and $\mu, \nu \in E^{-\infty}$, we will say $\mu \sim_{e} \nu$ if and only if there is $(g_{n})_{n<0} \subseteq G$ such that $d(g_{n}) = r(\mu_{n})$ and $g_{n}\cdot \mu_{n}....\mu_{-1} = \nu_{n}...\nu_{-1}$ for all $n<0$. Note that this is the same as asymptotic equivalence, except we do not require the sequence of groupoid elements to lie in a finite set.

\begin{prp}
Let $(G, E)$ be a finitely generated and contracting self-similar groupoid. Then, $\mathcal{C}(\Jj_{G, E}) = E^{-\infty}/\sim_{e}$.
\end{prp}

\begin{proof}
Let $q:E^{-\infty} \to \Jj_{G, E}$ be the quotient map. We show $q(\mu)\sim_{C}q(\nu)$ if and only if $\mu \sim_{e} \nu$.
\par
First, suppose $\mu, \nu \in E^{-\infty}$ are such that $\mu \nsim_{e} \nu$. Let $n<0 $ be such that $g\cdot \mu_{n}...\mu_{-1}\neq \nu_{n}...\nu_{-1}$ for all $g\in G$ such that $d(g) = r(\mu_{n})$. Then, the set 
\[
Z = \bigcup_{\{g\in G \mid d(g) = r(\mu_{n})\}}Z(g\cdot \mu_{n}..\mu_{-1}]
\] is clopen and does not contain $\nu$, which is also saturated with respect to the asymptotic equivalence relation. Therefore, $q(Z)\subseteq \Jj_{G, E}$ is a clopen set such that $q(\mu)\in q(Z)$ and $q(\nu)\notin q(Z)$. Hence, $q(\mu)\nsim_{C} q(\nu)$.

Suppose now $\mu,\nu \in E^{-\infty}$ are such that $\mu \sim_{e} \nu$. Let $V = s(E^{-\infty})$.
For $n<0$, let 
\[
Z_{n} = \bigcup_{\{g\in G \mid d(g) = r(x_{n}), \, c(g)\in V\}}Z(g\cdot \mu_{n}...\mu_{-1}].
\]
Then, $Z_{n-1}\subseteq Z_{n}$ and $\mu,\nu \in Z_{n}$ for all $n<0$. Denote $Z_{-\infty} = \bigcap_{n<0}Z_{n}$. We show $\Jj_{-\infty} \coloneqq q(Z_{-\infty})$ is connected.

Suppose $\Jj_{-\infty} = \Jj_{0}\cup\Jj_{1}$, where $\Jj_{0}$, $\Jj_{1}$ are non-empty, pairwise disjoint and clopen in the relative topology induced from $\Jj_{-\infty}$. Since $Z_{-\infty}$ is closed and saturated with respect to the asymptotic equivalence relation, $\Jj_{0}$ and $\Jj_{1}$ are also closed in $\Jj_{G, E}$. Let $X_{0} = q^{-1}(\Jj_{0})$ and $X_{1} = q^{-1}(\Jj_{1}).$ Let $d$ be an ultrametric metric on $E^{-\infty}$. Since $X_{0}$ and $X_{1}$ are disjoint compact sets and $Z_{-\infty} = X_{0}\cup X_{1}$, there is $N<0$ such that for all $n\leq N$, $VG \cdot \mu_{n}...\mu_{-1} = P_{0,n}\cup P_{1,n}$, where $P_{0,n}\cap P_{1,n} = \emptyset$ and $X_{0}\subseteq Z_{0,n} \coloneqq \bigcup_{y\in P_{0,n}}Z(y]$,  $X_{1}\subseteq Z_{1,n}\coloneqq \bigcup_{y\in P_{1,n}}Z(y]$. In particular, if $d(X_{0}, X_{1}) =\inf\{d(x_{0}, x_{1}) \mid x_{i}\in X_{i}\} = \alpha$ and $N$ satisfies $\text{sup}_{p\in E^{-N}}\text{diam}(Z(p]) < \alpha$, then for $n\leq N$, let 
\begin{align*}
P_{0,n} &= \{p\in VG \cdot \mu_{n}...\mu_{-1} \mid d(X_{0}, Z(p]) <\alpha\}.
\end{align*}
Since $d$ is an ultrametric, we have $Z(p]\cap X_{1}=\emptyset$ for all $p\in P_{0,n}$. Therefore, we may set $P_{1,n} = VG\cdot \mu_{n}...\mu_{-1}\setminus P_{0,n}$.
Since $G$ is finitely generated, so is the groupoid $G|_{V} = \{g\in G \mid d(g),c(g)\in V\}$. Let $F$ be a finite generating set for $G|_{V}$; i.e., $\bigcup_{n\in\mathbb{N}} F^{n} = G|_{V}$. For each $n<0$, choose $p^{0,n}\in P_{0,n}$ such that there is $f_{n}\in F$ with $p^{1,n} \coloneqq f_{n}\cdot p^{0,n}\in P_{1,n}$.

Since $r(p^{0,n}), r(p^{1,n})\in V$, there are infinite paths $x^{0,n},x^{1,n}\in E^{-\infty}$ such that $x^{0,n}\in Z(p^{0,n}]$ and $x^{1,n}\in Z(p^{1,n}]$ for all $n<N$. Let $(n_{k})_{k<0}$ be a decreasing sequence such that $(x^{0,n_{k}})_{k<0}$ and $(x^{1,n_{k}})_{k<0}$ converge to $x^{0}$ and $x^{1}$, respectively. Since $x^{0,n_{k}}, x^{1, n_{k}}\in Z_{n_{k}}$ for all $k<0$, we have $x^{0}, x^{1}\in Z_{-\infty}$. Let us show $x^{0}\in X_{0}$ and $x^{1}\in X_{1}$.

We have $d(y, x^{0,n_{k}})\geq \alpha$ for all $k<0$ and $y\in X_{1}$; for if not, then there is $K<0$, $y_{1}\in X_{1}$ and $y_{0}\in X_{0}$ such that $d(y_{0}, y_{1})\leq\max\{d(y_{0}, x^{0,n_{K}}), d(y_{1}, x^{0,n_{K}})\} < \alpha$, contradicting that $d(X_{0}, X_{1}) = \alpha$. Hence, we have $d(X_{1}, x^{0})\geq \alpha$. It follows that $x^{0}\in Z_{-\infty}\setminus X_{1} = X_{0}.$ By definition, $d(X_{0}, x^{1,n_{k}})\geq \alpha$ for all $k<0$, and so $x^{1}\in Z_{-\infty}\setminus X_{0} = X_{1}.$

Now, we show $x^{0}\sim_{ae}x^{1}$. Since $(x^{0, n_{k}})_{k<0}$ converges to $x^{0}$ and $(x^{1, n_{k}})_{k<0}$ converges to $x^{1}$, there is a non-increasing sequence $(m_{k})_{k<0}$ such that $m_{k}\geq n_{k}$, $\lim_{k\to -\infty}m_{k} = -\infty$ and $x_{m_{k}}^{i,n_{k}}....x^{{i,n_{k}}}_{-1} = x_{m_{k}}^{i}...x^{i}_{-1}$ for each $i\in\{0,1\}$ and $k<0$. If we denote $g_{m_{k}} =  f_{n_{k}}|_{p^{0,n_{k}}_{n_{k}}...p_{m_{k}-1}^{0,n_{k}}}$, for all $k<0$ then it follows that $g_{m_{k}}\cdot x_{m_{k}}^{0}...x^{0}_{-1} = x_{m_{k}}^{1}...x_{-1}^{1}$. Since $F$ is finite and $(G, E)$ is contracting, we have that $F' = \bigcup_{n\in\mathbb{N}}F|_{E^{n}}$ is finite. Define, for $m_{k} > n > m_{k+1}$, $g_{n} = g_{m_{k+1}}|_{x_{m_{k+1}}^{0}...x^{0}_{n+1}}$. Then, $(g_{n})_{n<0}\subseteq F'$ satisfies $d(g_{n}) = r(x_{n}^{0})$ and $g_{n}\cdot x_{n}^{0}...x^{0}_{-1} = x^{1}_{n}...x_{-1}^{1}$ for all $n<0$. Hence, $x^{0}\sim_{ae} x^{1}$.

We have shown $q(x^{0}) = q(x^{1})\in \Jj_{0}\cap \Jj_{1}$, which is a contradiction to the assumption that $\Jj_{0}\cap\Jj_{1}=\emptyset$. Hence, $\Jj_{-\infty}$ is connected and $q(\mu)\sim_{C} q(\nu)$.
\end{proof}

\begin{rmk}
    Let $\Cc_{(G, E)} = \Jj_{G, E}/\sim_{C}$ and $q_{C}:E^{-\infty}\to \Cc_{(G, E)}$ the quotient map. Define, for $\eta \in E^{n}$, $n\in\mathbb{N}$, the set 
\[
U_{\eta} = q_{C}\left(\bigcup_{\{g\in G \mid d(g) = r(\eta)\}}Z(g\cdot \eta]\right).
\]
These clopen sets form a basis for the quotient topology on $\Cc_{(G, E)}$. It is easy to see that $\mu \sim_{e} \nu$ implies $\sigma(\mu )\sim_{e}\sigma(\nu)$, and so there is an induced dynamical system $\sigma_{C}:\Cc_{(G, E)}\to \Cc_{(G, E)}$. It is not typically locally injective, but it is always an open mapping when $(G,E)$ is contracting.
\end{rmk}

\begin{prp}\label{prop:connectediff}
  Let $(G_B,E_A)$ be a KEP-action such that $B\in M_{N}(\{0,1\})$. For $\mu, \nu \in E^{-
\infty}_{A}$, $\mu \sim_{e} \nu$ if and only if $\mu = \nu$, or there is $K<0$ such that $s(\mu_{k}) = s(\nu_{k}) \coloneqq v_{k}$ and $B_{v_{k-1}, v_{k}} = 1$ for all $k\leq K$, and $B_{v_{K}, v_{K+1}} = 0$, $\mu_{K+1}...\mu_{-1} = \nu_{K+1}...\nu_{-1}$ if $K < 1$.
\end{prp}

\begin{proof}
Suppose $\mu \sim_{e} \nu$ and let $(g_{k})_{k<0}$ be a sequence of groupoid elements such that $d(g_{k}) = r(\mu_{k})$ and $g_{k}\cdot \mu_{k}...\mu_{-1} = \nu_{k}...\nu_{-1}$ for all $k<0$. Since $G_B$ is a group bundle, the action of it on $E^{*}_{A}$ preserves the range and source vertices of paths. Therefore, $s(\mu_{k}) = s(g_{k}\cdot \mu_{k}) = s(\nu_{k})$ for all $k<0$.

If $B_{v_{k-1}, v_{k}} = 0$, then we have $g_{k}\cdot \mu_{k} = \mu_{k}$ and $g|_{\mu_{k}} = v_{k}$, so that $\mu_{k}...\mu_{-1} = g \cdot \mu_{k}...\mu_{-1} = \nu_{k}...\nu_{-1}$. So, if $B_{v_{k-1}, v_{k}} = 0$ infinitely often, then $\mu = \nu$. Otherwise, there is $K<0$ such that $B_{v_{k-1}, v_{k}} = 1$ for $k\leq K$ and $B_{v_{K}, v_{K+1}} = 0$ if $K > 1$, in which case $\mu_{K+1}..\mu_{-1} = \nu_{K+1}...\nu_{-1}$.

We prove the reverse direction. For $k\leq K$, each $A_{v_{k-1}, v_{k}}$-odometer action is \textit{recurrent} in the sense that given $g\in G_{B}$ and $e, f\in E^{1}_{A}$ satisfying $d(g) = v_{k}$, $r(e) = r(f) = v_{k-1}$, and $s(e) = s(f) = v_{k}$, there is $h\in G_{B}$ such that $d(h) = v_{k-1}$, $h\cdot e = f$, and $h|_{e} = g$. It is then an easy induction argument to see the action of $\{g\in {G}_{B} \mid d(g) = v_{k-1}\}$ on $\{\mu = \mu_{k}...\mu_{K} \in E^{(K- k) +1}_{A} \mid r(\mu) = v_{k-1}, s(\mu_{j}) = v_{j}\text{ }\forall \, k\leq j\leq K\}$ is transitive.
\par
It follows that every path $\eta = ...\eta_{K-1}\eta_{K}$ such that $s(\eta_{k}) = s(\mu_{k})$ for all $k\leq K$ satisfies $\eta \sim_{e} ...\mu_{K-1}\mu_{K}$. If $K = 1$, then we are done. Otherwise, since $B_{v_{K}, v_{K+1}} = 0$, we have $\nu \coloneqq \eta \mu_{K+1}...\mu_{-1}\sim_{e} \mu$.
\end{proof}

Since the action of a KEP-action preserves the vertices of paths, there are factor maps  $\pi_{\Jj}:\Jj_{G_{B},E_{A}}\to E^{-\infty}_{C}$ and $\pi_{\Cc}:\Cc_{G_{B},E_{A}}\to E^{-\infty}_{C}$, where $C$ is the connectivity matrix of $E_{A}$. We can use this factor map to describe the connected components of $\Jj_{G_{B},E_{A}}$. The following fact is contained in the proof of Proposition \ref{prop:connectediff}.
\begin{cor}
 Let $(G_B,E_A)$ be a KEP-action such that $B\in M_{N}(\{0,1\})$ and $z\in \Jj_{(G_B,E_A)}$. Denote $\pi_{\Jj}(z) = v$. Then, $z$ is a connected component if $B_{v_{k-1}, v_{k}} = 0$ infinitely often.
\end{cor}
\begin{proof}
The fact that $z\in \Jj_{(G_B,E_A)}$ is a connected component if $B_{v_{k-1}, v_{k}} = 0$ infinitely often is contained in the proof of Proposition \ref{prop:connectediff}.
\end{proof}

 We now study the remaining case where $\pi_{\Jj}(z) = v$ satisfies $B_{v_{k-1},v_{k}} = 1$ eventually.

 Suppose first $B_{v_{k-1},v_{k}} = 1$ for all $k < 0$.
Consider the subset $X_{v} = \{\mu \in E_{A}^{-\infty} \mid s(\mu_{k}) = v_{k}, \, \forall k<0\}$. Let $\iota_{v}:X_{v}\to \mathcal{A}_{v} \coloneqq \Pi_{k<0}\{0,..., A_{v_{k-1}, v_{k}}-1\}$ be the natural identification, where we send $\mu = (...e_{v_{-3}, v_{-2}, i_{-2}}e_{v_{-2}, v_{-1}, i_{-1}})$ to $\iota_{v}(\mu) = (...i_{-2}i_{-1})$.

Define a mapping $C_{v}:\mathcal{A}_{v}\to \mathbb{T}^{1} = \mathbb{R}/\mathbb{Z}$ by sending $i = (...i_{-2}i_{-1})$ to $C_{v}(i) = \sum^{-\infty}_{k=-1} \frac{i_{k}}{A_{v[k,-1]}}.$ It is easy to see that $C_{v}:\mathcal{A}_{v}\to\mathbb{T}$ is a surjection if $\max_{k<0}A_{v[k,-1]}=\infty$; for $t\in [0,1)$ write $t_{0} = t$, $t_{0} = \frac{t_{-1} + i_{-1}}{A_{v_{-2},v_{-1}}}$ for some $t_{-1}\in [0,1)$ and $i_{-1}\in\{0,...,A_{v_{-2},v_{-1}}-1\}$ and inductively $t_{k-1} = \frac{t_{k} + i_{k}}{A_{v_{k-1},v_{k}}}$, for $t_{k}\in [0,1)$ and $i_{k}\in\{0,...,A_{v_{k-1},v_{k}}-1\}$. Since $\max_{k<0}A_{v[k,-1]}=\infty$, we have $C_{v}(...i_{-2},i_{-1}) = t$.
\par
If $(G_{B}, E_{A})$ is regular, the case where $\max_{k<0}A_{v[k,-1]}<\infty$ can only happen if $A_{v_{k-1}, v_{k}} = 1$ for all $k\in\mathbb{N}$. In this case, $\mathcal{A}_{v}$ is a single point.

\begin{prp}\label{ae_wrt_connected}
Let $(G_B, E_A)$ be a regular KEP-action such that $B\in M_{N}(\{0,1\})$. Suppose $v\in E_{C}^{-\infty}$ satisfies $B_{v_{k-1},v_{k}} = 1$ for all $k\in\mathbb{N}$. For $\mu, \nu \in\mathcal{A}_{v}$, $\mu \sim_{ae} \nu$ if and only if $C_{v}(\mu) = C_{v}(\nu)$.
\end{prp}

\begin{proof}
The proposition is trivial (by regularity) if $\max_{k<0}A_{v[k,-1]}<\infty$, so we assume $\max_{k<0}A_{v[k,-1]}=\infty$.
Given $t\in [0,1)$, there is a unique choice for $(...i_{-2}i_{-1})$ if and only if $t_{k}\neq 0$ for all $k<0$.

If $t_{k} = 0$, for some $k<0$ then if we let $K\geq k$ be the first number such that $i_{K}\neq 0$, then $C_{v}(...00i_{K}...i_{-1}) = C_{v}(...(A_{v_{K-2},v_{K-1}}-1)(i_{K}-1)...i_{-1})$, and these are the only two choices. 

Suppose $\mu\sim_{a.e.}\nu$ for $\mu,\nu\in X_{v}$, and let
$\{g_{k}\}_{k<0}$ and $F\subseteq \mathbb{N}$ be a finite set satisfy $g_{k}\cdot \mu_{k}...\mu_{-1} = \nu_{k}...\nu_{-1}$ for all $k < 0$ and $g_{k} = a^{m_{k}}_{r(\mu_{k})}$ for $m_{k}\in F$. By \eqref{induc on n} and $\max_{k<0}A_{v[k,-1]}=\infty$, there is $K\in\mathbb{N}$ such that $g_{k}|_{\mu[k,k+K-1]} = a^{l_{k}}_{r(\mu_{k+K})}$ for some $l_{k}\in\{-1,0,1\}$. Since $\{g\in G_{B}\mid  g = a_{i}^{l},l\in\{-1,0,1\}\}$ is invariant under the restriction map, it follows that we may assume $g_{k} = a_{r(\mu_{k})}^{m_{k}}$ for $m_{k}\in\{-1,0,1\}$. Further invariance conditions imply either $m_{k}\geq 0$ for all $k <0$ or $m_{k}\leq 0$ for all $k < 0$. It is then routine to see $\mu \sim_{ae} \nu$ if and only if $\iota(\mu) = \iota(\nu)$ or $\{\iota(\mu), \iota(\nu)\} = \{(...00i_{k}...i_{-1}),(...(A_{v_{k-2},v_{k-1}}-1)(i_{k}-1)...i_{-1})\}$ for some $k\in\mathbb{N}$.
\end{proof}

We have for $\mu \in\mathcal{A}_{v}$, 
\[
C_{v}(\mu)^{A_{v_{-2},v_{-1}}} = A_{v_{-2},v_{-1}}\sum^{-\infty}_{k=-1}\frac{i_{n}}{A_{v[k,-1]}} = \sum^{-\infty}_{k=-2}\frac{i_{k}}{A_{v[k,-2]}} = C_{\sigma(v)}(\sigma(\mu)).
\]
Therefore, when the connected components above the paths $v$ and $\sigma(v)$ in the connectivity graph $E_{C}$ of $E_{A}$ are identified with the circle, the dynamics becomes $z\to z^{A_{v_{-2}, v_{-1}}}$. We are now able to summarise the results of this section into a theorem.

\begin{thm}
\label{thm:description}
Let $(G_B,E_A)$ be a regular KEP-action such that $B\in M_{N}(\{0,1\})$. Then, a connected component of $\Jj_{G_{B}, E_{A}}$ is either a point or a circle.

Let $\pi:E^{-\infty}_{A}\to E^{-\infty}_{C}$ and $\pi_{\Jj}:\Jj_{G_{B}, E_{A}}\to E^{-\infty}_{C}$ be the induced factor maps, where $C$ is the connectivity matrix for the graph $E_{A}$. For $v \in E_C^{-\infty}$, let $K(v) = -\min\{k<0 \mid  B_{v_{k-1}, v_{k}} = 0\}$. Set $K(v) = 0$ if $\{k<0 \mid  B_{v_{k-1}, v_{k}} = 0\} = \varnothing$.
\begin{enumerate}
    \item If $K(v) = \infty$, then $\pi_{\Jj}^{-1}(v)\simeq \pi^{-1}(v)$. Under this identification, $\tilde{\sigma}$ is the shift $\sigma:\pi^{-1}(v)\to \pi^{-1}(\sigma(v))$.
    
    \item If $K(v) = 0$ and $\max_{k<0}A_{v[k,-1]} = \infty$, then $\pi_{\Jj}^{-1}(v)\simeq \mathbb{T}_{v}$. Under this identification, $\tilde{\sigma}$ is $z^{n}:\mathbb{T}_{v}\to \mathbb{T}_{\sigma(v)},$ $n = A_{v_{-2},v_{-1}}$.
    
    \item If $K(v) = 0$ and $\max_{k<0}A_{v[k,-1]} < \infty$, then $\pi^{-1}(v)$ is a single point.
    
    \item if $ 0 < K(v) < \infty$, then $\pi_{\Jj}^{-1}(v)\simeq \pi^{-1}_{\Jj}(\sigma^{K(v)}(v))\times \pi^{-1}(v_{K-1}...v_{1})$ and cases $(2)$ and $(3)$ apply to describe $\pi^{-1}_{\Jj}(\sigma^{K(v)}(v))$. Under this identification, $\tilde{\sigma}$ is $\text{id}\times \sigma:\pi^{-1}_{\Jj}(\sigma^{K(v)}(v))\times \pi^{-1}(v_{K-1}...v_{1})\to \pi^{-1}_{\Jj}(\sigma^{K(v)}(v))\times \pi^{-1}(v_{K-1}...v_{2})$.
\end{enumerate}

\end{thm}

\begin{exm}
    The description in Theorem \ref{thm:description} does not extend when $B$ takes values different than $0$ or $1$. For instance, if $A = (3)$ and $B = (2)$, then $\Jj_{G_B,E_A}$ is not homeomorphic to the circle. For if $\Jj_{(G_B,E_A)}\simeq \mathbb{T}^{1}$, then $(\tilde{\sigma}, \Jj_{G_B,E_A})$ is conjugate to either $(\alpha^{-3}, \mathbb{T}^{1})$ or $(\alpha^{3}, \mathbb{T}^{1})$, where $\alpha:\mathbb{T}^{1}\to \mathbb{T}^{1}$ is a homeomorphism isotopic to the identity. Hence, the K-theory of the $C^{*}$-algebra associated to $(\tilde{\sigma}, \Jj_{G_B,E_A})$ is isomorphic to the K-theory associated to $(z^{3}, \mathbb{T}^{1})$ or $(z^{-3}, \mathbb{T}^{1})$, which by \cite[Theorem~3]{Hume} is either  $(\mathbb{Z}/2\mathbb{Z}\oplus\mathbb{Z},\mathbb{Z})$ or $(\mathbb{Z}/2\mathbb{Z}, \mathbb{Z}/2\mathbb{Z})$. However, By \cite{BBGHSW}, we have $\mathcal{O}_{(\tilde{\sigma}, \Jj_{G_B,E_A})}\simeq \mathcal{O}_{A,B}$ and Katsura shows in \cite[Example~A.6]{Katsura:class_I} that the K-theory of $\mathcal{O}_{A,B}$ is $(\mathbb{Z}/2\mathbb{Z}, 0)$. \qed
\end{exm}

\section{Planar embedding of Putnam's spaces}

In \cite[Question~7.10]{Putnam:Binary}, Putnam asked when $\mathcal{J}_{\xi}=E^{-\infty}/\sim_{\xi}$ embeds into the plane. In this section, we prove $\Jj_{\xi}$ always embeds into the plane by using our description of $\mathcal{J}_{\xi}$ as the limit space of a regular KEP-action. In particular, Theorem \ref{thm:description} gave a topological description of the limit space when $B \in M_N(\{0,1\})$ as a Cantor set bundle of circles (with the convention that a point is a circle of radius zero). We use this as inspiration to define an embedding from the limit space to the complex plane whenever $B\in M_N(\{0,1\})$ and $(G_{B}, E_{A})$ is regular. See Corollary \ref{regular_characterization_01} for a characterisation of regularity in terms of $A$ and $B$. Notice that Section \ref{Katsura models Putnam} proved that the KEP-actions from embedding pairs always have the property that $A \in M_N(\{0,1,2\})$ and $B=(B_{ij})$ where $B_{ij}=\max\{A_{ij}-1,0\}$. Thus, Putnam's question will be answered by the following more general result.

\begin{thm}\label{Main Embedding Thm}
Let $(G_B,E_A)$ be a regular KEP-action such that $B\in M_N(\{0,1\})$. Then, there is a continuous injection $\zeta:\Jj_{G_B,E_A}\to \mathbb{C}$.
\end{thm}

\begin{cor}
    Let $\xi = \xi^{0},\xi^{1}:H\to E$ be an embedding pair. Then, there is a continuous injection $\zeta:\mathcal{J}_{\xi} \to \mathbb{C}$.
\end{cor}

To prove these results, we make several definitions in order to define the map $\zeta:\Jj_{G_B,E_A}\to \mathbb{C}$. First, we will assume $\|A\|_{\text{max}}=\max_{i,j} |a_{i,j}| > 1$, otherwise $\Jj_{G_{B},E_{A}}\simeq E_{A}^{-\infty}$ and it is a classical fact $E^{-\infty}_{A}$ embeds into $\mathbb{C}.$

Recall that we are working with infinite paths $\mu \in E_{A}^{-\infty}$ with edges labelled by negative integers $\mu= \ldots \mu_{-2} \mu_{-1}$.

For negative integers $m,n$ such that $m\leq n$, let $[m,n] = \{k\in\mathbb{Z}\mid m\leq k\leq n\}$, which we will call an \textit{interval}. We will also consider infinite intervals $[-\infty, n] = \{k\in\mathbb{Z}\mid k\leq n\}$. We will denote the collection of intervals by $\mathcal{I}$, and if $I\in\mathcal{I}$, then we will let $I^{-},I^{+}\in\mathbb{Z}\cup\{-\infty\}$ be such that $I = [I^{-}, I^{+}]$. 

For $e_{i,j,k} \in E^{1}_{A}$, we let $A(e_{i,j,k}) = A_{i,j}$, $B(e_{i,j,k}) = B_{i,j}$ and $\#(e_{i,j,k}) = k$. If $\mu \in E_{A}^{-\infty}$, then let $\mathcal{I}^{1}(\mu)$ be the collection of intervals $I$ such that $B(\mu_{j}) = 1$ for all $j\in I$, and is maximal with respect to this property. We will call an interval in $\mathcal{I}^{1}$ \textit{type 1}. Similarly, let $\mathcal{I}^{0}(\mu)$ denote the maximal intervals $I$ satisfying the property $B(\mu_{j}) = 0$ for all $j\in I$ and call these intervals \textit{type 0}. Then, $\mathcal{I}(\mu) = \mathcal{I}^{0}(\mu)\cup \mathcal{I}^{1}(\mu)$ is a collection of pairwise disjoint intervals such that $\bigcup_{I\in \mathcal{I}(\mu)}I = [-\infty, -1]$. 

We aim to define an embedding map $\zeta:E^{-\infty}_{A}\to\mathbb{C}$, which has several components. For $-n\in\mathbb{N}\cup\{\infty\}$ and  $\mu = \mu_{-n}...\mu_{-1} \in E_{A}^{n}$,  we adapt \eqref{path in matrix} by defining 
\[
A^0_{\mu} = \prod_{j=-1}^{-n} (A(\mu_{j})+1) \quad \text{ and } \quad A^1_{\mu} = \prod_{j=-1}^{-n} A(\mu_j),
\]
in order to define
\[
\theta^0(\mu) = \sum_{j=-1}^{-n}\frac{\#(\mu_{j})}{A^0_{\mu[j,-1]}} \quad \text{ and } \quad \theta^1(\mu) = \sum_{j=-1}^{-n}\frac{\#(\mu_{j})}{A^1_{\mu[j,-1]}}.
\]
Moreover, let $M=\|A\|_{\text{max}}=\max_{i,j} |a_{i,j}|$ and $R=M(N+1)$, and define
\[
\Omega(\mu) = \sum_{j=-1}^{-n} s(\mu_{j})R^{j} + r(\mu_{-n})R^{-(n+1)}.
\]

For an interval $I \subseteq [-\infty,-1]$ and $\mu \in E_{A}^{-\infty}$, let $\mu_{I} = \mu[I^-,I^+]$. Define $\zeta:E^{-\infty}_{A}\to\mathbb{C}$ by
\begin{equation}\label{embedding_map}
\zeta(\mu) = \\
\sum_{I\in\mathcal{I}^{0}(\mu)}R^{3I^+}\Omega(\mu_{I})e^{2\pi i \left(\theta^0(\mu_{I})\right)} + \sum_{I\in\mathcal{I}^{1}(\mu)}R^{3I^+}\Omega(\mu_{I})e^{2\pi i \left(\theta^0(\mu_{I})\right)}.
\end{equation}
The idea of breaking apart ``$A$-ary'' expansion along the intervals in $\mathcal{I}(\mu)$ was inspired by Putnam's construction of a metric for $\Jj_{\xi}$, where this idea appears in a more basic form. 

For $k < 0$ and $I\in\mathcal{I}$, denote $I\cap [-k,-1]=:I_{k}$. Observe that $\zeta$ is continuous, as it is the uniform limit of $(\zeta_{k}:E^{-\infty}_{A}\to \mathbb{C})_{k<0}$, defined for $ \mu \in E_{A}^{-\infty}$ as 
\begin{align*}
\zeta_{k}(\mu) = \sum_{I\in\mathcal{I}^{0}(\mu):\text{ } k\leq I^{+}}R^{3I^+}\Omega(\mu_{I_{k}})e^{2\pi i \left(\theta^0(\mu_{I_{k}})\right)}
+ \sum_{I\in\mathcal{I}^{1}(\mu):\text{ } k\leq I^{+}}R^{3I^+}\Omega(\mu_{I_{k}})e^{2\pi i \left(\theta^1(\mu_{I_{k}})\right)},
\end{align*}
which is continuous as it only depends on $\mu_{[k,-1]}$. 

Before continuing the proof, we consider an example that gives a feeling as to how the embedding works.

\begin{center}
\begin{figure}
\begin{tikzpicture}[scale=1.0]
\node at (0,0) {$2$};
\node[vertex] (vertexe) at (0,0)   {$\quad$}
	edge [->,>=Straight Barb,out=-30,in=30,loop] node[left,pos=0.5]{$\scriptstyle e_{2,2,0}$} (vertexe)
	edge [->,>=Straight Barb,out=-50,in=50,loop,looseness=10] node[right,pos=0.5]{$\scriptstyle e_{2,2,1}$} (vertexe);
\node at (-3,0) {$1$};
\node[vertex] (vertex-a) at (-3,0)   {$\quad$}
	edge [->,>=Straight Barb,out=20,in=160] node[below,swap,pos=0.5]{$\scriptstyle e_{2,1,0}$} (vertexe)
	edge [->,>=Straight Barb,out=45,in=135] node[below,swap,pos=0.5]{$\scriptstyle e_{2,1,1}$} (vertexe)
	edge [->,>=Straight Barb,out=70,in=110] node[above,swap,pos=0.5]{$\scriptstyle e_{2,1,2}$} (vertexe)
	edge [->,>=Straight Barb,out=210,in=150,loop] node[right,pos=0.5]{$\scriptstyle e_{1,1,1}$} (vertex-a)
	edge [->,>=Straight Barb,out=230,in=130,loop,looseness=10] node[left,pos=0.5]{$\scriptstyle e_{1,1,0}$} (vertex-a)
	edge [<-,>=Straight Barb,out=330,in=210] node[above,swap,pos=0.5]{$\scriptstyle e_{1,2,1}$} (vertexe)
	edge [<-,>=Straight Barb,out=300,in=240] node[below,swap,pos=0.5]{$\scriptstyle e_{1,2,0}$} (vertexe);
\end{tikzpicture}
\caption{The graph $E_A$ specified by the adjacency matrix $A$ from Example \ref{embedding_ex}.}
\label{Katsura_graph_embedding_ex}
\end{figure}
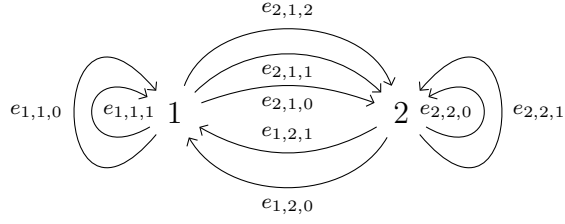
\end{center}

\begin{exm}\label{embedding_ex}
Let $(G_B,E_A)$ be the KEP-action defined by 
\begin{equation}\label{Katsura_matrices_rmkexm3}
A=\left(\begin{matrix} 2 & 2 \\ 3 & 2 \end{matrix}\right) \quad \text{ and } \quad B=\left(\begin{matrix} 1 & 1 \\ 0 & 1 \end{matrix}\right).
\end{equation}
Then $A$ is the adjacency matrix for the graph $E_A$ depicted in Figure \ref{Katsura_graph_embedding_ex}.

Notice that there are classical odometers at vertices $1$ and $2$, along with a countably infinite number of odometers with range $1$ and eventual source at vertex $2$. On the other hand, we only have the one odometer action with range $2$, since any groupoid element $a_2^m$ restricts to the unit $1$ through the edges $e_{2,1,i}$ for $i=0,1,2$.

Consider the collection of paths with each edge having range and source $1$,
\[
O_{1,1}=\{\ldots e_{1,1,{m_{-3}}} e_{1,1,{m_{-2}}} e_{1,1,{m_{-1}}} \mid  m_i \in \{0,1\} \text{ for all }i<0\}
\]
We have that $\mathcal{I}^{1}(\mu)=[-\infty,-1]$ for all $\mu \in O_{1,1}$ and, using $M=2$, $N=2$, and $R=M(N+1)=6$, we compute
\begin{align*}
\Omega(\mu) &=\sum_{j=-1}^{-\infty} 6^{j} s(\mu_j)=\sum_{j=1}^\infty \frac{1}{6^{j}}=\frac{1}{5} \quad \text{ and } \quad
\theta(\mu) =\sum_{j=-1}^{-\infty} 2^{j}\#(\mu_j)=\sum_{j=1}^\infty \frac{m_{-j}}{2^{j}}.
\end{align*}
Thus,
\begin{align*}
\zeta(\mu) &=\frac{1}{5 \cdot 6^3} e^{2\pi i(\theta(\mu) + 0)} = \frac{1}{1080} e^{2\pi i \theta(\mu)}
\end{align*}
Observe that $O_{1,1}$ is in bijective correspondence with binary representations of numbers in $[0,1]$, reading right to left, and two decimal expansions are equal exactly when the paths in $O_{1,1}$ are asymptotically equivalent. Thus, the image of $\zeta:O_{1,1} \to \CC$ is the circle centred at the origin with radius $1/1080$. Moreover, since $O_{1,1}$ is a classical odometer, $\mathcal{J}_{G_B,E_A}$ restricted to $O_{1,1}$ is a circle \cite[p.72]{Nekrashevych:Self-similar} and the embedding is bijective.

Similar computations show that
\[
O_{2,2}=\{\ldots e_{2,2,{m_{-3}}} e_{2,2,{m_{-2}}} e_{2,2,{m_{-1}}} \mid  m_i \in \{0,1\} \text{ for all }i<0\}
\]
maps to the circle centred at the origin with radius $1/540$. Moreover, consider
\[
O_{n}=\{\ldots e_{1,1,{m_{-n-2}}} e_{1,1,{m_{-n-1}}} e_{1,2,m_{-n}} e_{2,2,{m_{-n+1}}} \ldots e_{2,2,{m_{-1}}} \mid  m_i \in \{0,1\} \text{ for all } i<0\}.
\]
For $\nu \in O_n$ we compute
\begin{align*}
\Omega(\nu) &=\sum_{j=-1}^{-\infty} 6^{j}s(\nu_j)=\sum_{j=1}^n \frac{2}{6^{j}}+\sum_{j=n+1}^\infty \frac{1}{6^{j}}=\frac{2-\frac{1}{6^n}}{5}  \\
\end{align*}
Thus,
\begin{align*}
\zeta(\nu) &=\frac{2-\frac{1}{6^n}}{5\cdot 6^3} e^{2\pi i\theta(\nu)} =\left(\frac{1}{540}-\frac{1}{1080 \cdot 6^n}\right) e^{2\pi i\theta(\nu)}.
\end{align*}
So, for each $n \in \NN$ we have a circle centred at the origin of radius $\frac{1}{540}-\frac{1}{1080 \cdot 6^n}$.

\begin{center}
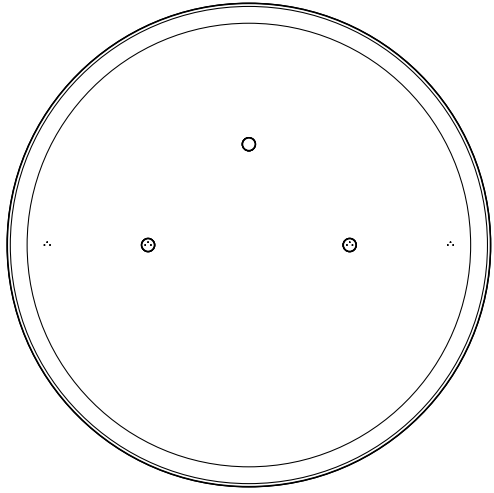
\begin{figure}
\begin{tikzpicture}[scale=8]
\begin{scope}[xshift=-6.5cm,yshift=-1.25cm]
\foreach \n/\nn in {1, 2, 3, 4,5} {
\draw (0,0) circle ({(1/(5))*(2-(1/6^\nn))});}
\foreach \n/\nn in {1, 2, 3, 4,5} {
\draw (1/6,0) circle ({(1/(36*5))*(2-(1/6^\nn))});}
\foreach \n/\nn in {1, 2, 3, 4,5} {
\draw (-1/6,0) circle ({(1/(36*5))*(2-(1/6^\nn))});}
\foreach \n/\nn in {1, 2, 3, 4,5} {
\draw (0,1/6) circle ({(1/(36*5))*(2-(1/6^\nn))});}
\foreach \n/\nn in {1, 2, 3, 4,5} {
\draw ({1/6+1/216},0) circle ({(1/(36*36*5))*(2-(1/6^\nn))});}
\foreach \n/\nn in {1, 2, 3, 4,5} {
\draw (1/6,1/216) circle ({(1/(36*36*5))*(2-(1/6^\nn))});}
\foreach \n/\nn in {1, 2, 3, 4,5} {
\draw ({1/6-1/216},0) circle ({(1/(36*36*5))*(2-(1/6^\nn))});}
\foreach \n/\nn in {1, 2, 3, 4,5} {
\draw ({-1/6+1/216},0) circle ({(1/(36*36*5))*(2-(1/6^\nn))});}
\foreach \n/\nn in {1, 2, 3, 4,5} {
\draw (-1/6,1/216) circle ({(1/(36*36*5))*(2-(1/6^\nn))});}
\foreach \n/\nn in {1, 2, 3, 4,5} {
\draw ({-1/6-1/216},0) circle ({(1/(36*36*5))*(2-(1/6^\nn))});}
\foreach \n/\nn in {1, 2, 3, 4,5} {
\draw ({2/6+1/216},0) circle ({(1/(36*36*5))*(2-(1/6^\nn))});}
\foreach \n/\nn in {1, 2, 3, 4,5} {
\draw (2/6,1/216) circle ({(1/(36*36*5))*(2-(1/6^\nn))});}
\foreach \n/\nn in {1, 2, 3, 4,5} {
\draw ({2/6-1/216},0) circle ({(1/(36*36*5))*(2-(1/6^\nn))});}
\foreach \n/\nn in {1, 2, 3, 4,5} {
\draw ({-2/6+1/216},0) circle ({(1/(36*36*5))*(2-(1/6^\nn))});}
\foreach \n/\nn in {1, 2, 3, 4,5} {
\draw (-2/6,1/216) circle ({(1/(36*36*5))*(2-(1/6^\nn))});}
\foreach \n/\nn in {1, 2, 3, 4,5} {
\draw ({-2/6-1/216},0) circle ({(1/(36*36*5))*(2-(1/6^\nn))});}
\end{scope}
\end{tikzpicture}
\caption{The embedding $\zeta:\mathcal{J}_{G_B,E_A} \to \mathbb{C}$ for Example \ref{embedding_ex}. The outer circles are centred at the origin with radius $\frac{1}{540}-\frac{1}{1080 \cdot 6^n}$.The other visible circles are scaled copies of the outer circles and continue ad infinitum.}
\label{Katsura_graph_embedding_ex2}
\end{figure}
\end{center}

Now, we consider a collection that does not give a circle centred at the origin. Consider the collection
\[
\mathcal{P}= \{\nu e_{2,1,m_{-1}} \mid m_{-1} \in \{0,1,2\} \text{ and } \nu \in \mathcal{O}_{n} \}.
\]
For $\eta \in \mathcal{P}$ we compute
\[
\Omega(\eta_{[-1,-1]}) =\frac{1}{6} \quad \text{ and } \quad
\theta(\eta_{[-1,-1]}) = \frac{\#(\mu_{-1})}{3+1} =\{0,1/4,1/2\}.
\]
Thus, the image of $\zeta(\mathcal{P})$ consists of circles of radius $\frac{1}{6^3}\left(\frac{1}{540}-\frac{1}{1080 \cdot 6^n}\right)$ centred at $\frac{1}{6^4}$, $\frac{i}{6^4}$, and $\frac{-1}{6^4}$.

See Figure \ref{Katsura_graph_embedding_ex2} for the visible image of $\zeta:\mathcal{J}_{G_B,E_A} \to \mathbb{C}$. \qed
\end{exm}

We must show that for $\mu, \nu \in E_{A}^{-}$, we have $\zeta(\mu) = \zeta(\nu)$ if and only if $\mu \sim_{ae} \nu$. One direction is easy.

\begin{prp}
\label{easy}
Let $(G_B,E_A)$ be a regular KEP-action such that $B\in M_N(\{0,1\})$.  If $\mu, \nu \in E_{A}^{-\infty}$ satisfy $\mu \sim_{ae} \nu$, then $\zeta(\mu) = \zeta(\nu)$.
\end{prp}

\begin{proof}
Proposition \ref{ae_wrt_connected} shows that $\mu \sim_{ae} \nu$  for $\mu\neq \nu$ if and only if there is $k<0$ and $I = [-\infty, k] \in \mathcal{I}^{1}(\mu)\cap\mathcal{I}^{1}(\nu)$ such that $\mu_{[k+1,-1]} = \nu_{[k+1,-1]}$, $s(\mu_{j}) = s(\nu_{j})=:v_{j}$, $B_{v_{j-1}, v_{j}} = 1 $ for all $j\leq k$, and $C_{v}(\mu_{[-\infty,k]}) = C_{v}(\nu_{[-\infty,k]})$.

Since $\mu_{[k+1,-1]} = \nu_{[k+1,-1]}$ and $s(\mu_{j}) = s(\nu_{j})$ for all $j\leq k$, we have that $\mathcal{I}^{0}(\mu) = \mathcal{I}^{0}(\nu)$, $\mathcal{I}^{1}(\mu) = \mathcal{I}^{1}(\nu)$, and $\Omega(\mu_{I}) = \Omega(\nu_{I})$ for all $I\in \mathcal{I}(\mu)$, as well as $\theta^i(\mu_{I}) = \theta^i(\nu_{I})$ for all $I\in\mathcal{I}^{i}(\mu)\setminus [-\infty, k]$, for $i\in \{0,1\}$. The equality $\theta^1(\mu_{[-\infty, k]}) = \theta^1(\nu_{[-\infty, k]})$ follows from $C_{v} = \theta^{1}|_{\mathcal{A}_{v}}$ and the above paragraph. All the above equalities then imply $\zeta(\mu) = \zeta(\nu)$.
\end{proof}

To prove the converse direction requires a more careful analysis that we undertake through a series of lemmas.

\begin{lem}
\label{lem:0}
Let $(G_B,E_A)$ be a regular KEP-action such that $B\in M_N(\{0,1\})$. Suppose $\mu, \nu$ are in $E^{-\infty}_{A}$ and let $J \in \mathcal{I}(\mu)$ and $K \in \mathcal{I}(\nu)$ be the intervals containing $-1$, respectively. If $J^- \leq K^-$, $-\infty< K^-$ and $\Omega(\mu_{K})\neq \Omega(\nu_{K})$, then $\zeta(\mu) \neq \zeta(\nu)$.
\end{lem}

\begin{proof}
We can write $R^{3}\zeta(\mu) = w\Omega(\mu_{K}) + r$ and $R^{3}\zeta(\nu) = z\Omega(\nu_{K}) + s$ for some $w,z \in\mathbb{T}$ and $r,s\in\mathbb{C}$ such that $|r|, |s| \leq \frac{R^{(K^{-}-1)}N}{(R - 1)}$. For any $0 < p,q\in\mathbb{R}$, we have $|wp - zq|\geq |p-q|$. It follows that 
\begin{align*}
R^{3}|\zeta(\mu) &- \zeta(\nu)| \geq |\Omega(\mu_{K}) - \Omega(\nu_{K})| - \frac{R^{(K^{-}-1)}2N}{(R - 1)} \\
&\geq R^{(K^{-}-1)} - \frac{R^{(K^{-}-1)}2N}{(R - 1)} = R^{(K^{-}-1)}\left(1 - \frac{2N}{R-1}\right) =:(*).
\end{align*}
 As $R=M(N+1)$ and $M \geq 2$, we have 
 \begin{align*}
& (*) = R^{(K^{-}-1)}\left(1 - \frac{2}{M}\frac{1}{1 + \frac{1}{N} - \frac{1}{MN}}\right)\geq R^{(K^{-}-1)}\left(1 - \frac{1}{1 + \frac{1}{N} - \frac{1}{MN}}\right) > 0. \qedhere
\end{align*}
\end{proof}

\begin{lem}
\label{lem:1}
Let $(G_B,E_A)$ be a regular KEP-action such that $B\in M_N(\{0,1\})$. Suppose $\mu, \nu$ are in $E^{-\infty}_{A}$ and let $J \in \mathcal{I}(\mu)$ and $K \in \mathcal{I}(\nu)$ be the intervals containing $-1$, respectively. If $J^- < K^-$, then $\zeta(\mu) \neq \zeta(\nu)$.
\end{lem}

\begin{proof}
If $\Omega(\mu_{K}) \neq \Omega(\nu_{K})$, then Lemma \ref{lem:0} implies $\zeta(\mu) \neq \zeta(\nu)$, so assume $\Omega(\mu_{K}) = \Omega(\nu_{K})$. Therefore, $J$ and $K$ are of the same type. We denote this type by $i \in \{0,1\}$. We have 
\[
R^{3}|\zeta(\mu)| \geq \left|\Omega(\mu_{J})e^{2\pi i \left(\theta^i(\mu_{J})\right)}\right| - R^{3J^{-}}\sum^{\infty}_{j=1}\frac{N}{R^{j}} = \Omega(\mu_{J}) - \frac{R^{3J^{-}}N}{R-1}
\]
and 
\[
R^{3}|\zeta(\nu)| \leq \left|\Omega(\nu_{K})e^{2\pi i \left(\theta^i(\nu_{K})\right)}\right| + R^{3K^{-}}\sum^{\infty}_{j=1}\frac{N}{R^{j}} = \Omega(\nu_{K}) + \frac{R^{3K^{-}}N}{R-1}
\]
Therefore, 
\[
R^3|\zeta(\mu)|-R^{3}|\zeta(\nu)| \geq \Omega(\mu_{J})-\Omega(\nu_{K})-\frac{N}{R-1}\left(R^{3J^{-}} + R^{3K^{-}} \right)
\]
Since $\Omega(\mu_{K}) = \Omega(\nu_{K})$, we have 
\[
\Omega(\mu_{J}) - \Omega(\nu_{K}) \geq \sum_{j=-K^{-} +2}^{-J^{-}+1}\frac{1}{R^{j}} \geq \frac{R^{K^{-}-1}}{2(R-1)},
\]
so to show $|\zeta(\mu)| > |\zeta(\nu)|$, it suffices to show 
\[
\frac{R^{K^{-}-1}}{2(R-1)} > \frac{N}{R-1}\left(R^{3J^{-}} + R^{3K^{-}} \right).
\]
Or equivalently, by dividing the above inequality by the left hand side, show $1 > 2N(R^{3J^{-} - K^{-}+1} + R^{2K^{-}+1}).$ Using $3J^{-} - K^{-} + 1\leq -3$ and $K^{-}\leq -1$, we have $2N(R^{-3} + R^{-1})\geq 2N(R^{3J^{-} - K^{-}+1} + R^{2K^{-}+1})$, so it suffices to show $R^{3} > 2N(1 +R^{2})$. Using $R = M(N+1)$,  and $M\geq 2$, we have
\[R^{3} = MNR^{2} + R^{2} \geq 2NR^{2} + R^{2} > 2NR^{2} + 2N = 2N(1 + R^{2}). \qedhere\]
\end{proof}

\begin{lem}
\label{lem:4}
Let $(G_B,E_A)$ be a regular KEP-action such that $B\in M_N(\{0,1\})$. Suppose $\mu \neq \nu$ are in $E^{-\infty}_{A}$ such that $\zeta(\mu)=\zeta(\nu)$. Let $j<0$ be the largest number such that $\mu_j \neq \nu_j$ and let $J_{1}(\mu) \in \mathcal{I}(\mu)$ and $J_{1}(\nu) \in \mathcal{I}(\nu)$ be the intervals containing $j$. Then $J_{1}(\mu)^+=J_{1}(\nu)^+$. 
\end{lem}

\begin{proof}
Let $k>j$ be the smallest number such that $k = I_{\mu}^{-}$ and $k = I_\nu^{-}$ for some $I_\mu \in\mathcal{I}(\mu)$ and $I_\nu \in\mathcal{I}(\nu)$. Since $\mu_{m} = \nu_{m}$, for all $m\geq k$, it follows that $I_\mu =I_\nu$. Thus, there exists $z\in\mathbb{C}$ such that $\zeta(\mu) = z + R^{3k}\zeta(\sigma^{|k|}(\mu))$ and $\zeta(\nu) = z + R^{3k}\zeta(\sigma^{|k|}(\nu))$. Hence, $\zeta(\sigma^{|k|}(\mu)) = \zeta(\sigma^{|k|}(\nu))$. Denote $\mu' = \sigma^{|k|}(\mu)$ and $\nu' = \sigma^{|k|}(\nu)$.  Let's now show $J_{1}(\mu)^{+} = k-1 = J_{1}(\nu)^{+}$, or equivalently, $J_{1}(\mu')^{+} = -1 = J_{1}(\nu')^{+}$.  Let $K_{1}(\mu')\in\mathcal{I}(\mu')$ and $K_{1}(\nu')\in\mathcal{I}(\nu')$ be the intervals containing $-1$.

By minimality of $k$, either $j-k\in K_{1}(\mu')$ or $j-k\in K_{1}(\nu')$ or $[K_{1}(\mu')]^{-}\neq [K_{1}(\nu')]^{-}$. Hence, we have either
\begin{enumerate}
\item $j-k\in K_{1}(\nu')\cap K_{1}(\mu')$,
\item $j-k\notin K_{1}(\nu')\cap K_{1}(\mu')$ and $j-k\in K_{1}(\nu')\cup K_{1}(\mu')$, or
\item $[K_{1}(\mu')]^{-}\neq [K_{1}(\nu')]^{-}$ .
\end{enumerate}

Let us confirm the lemma in each case.

$1:$ If $j-k\in K_{1}(\nu')\cap K_{1}(\mu')$, then $J_{1}(\mu') = K_{1}(\mu')$, and $J_{1}(\nu') = K_{1}(\nu')$. Hence, $J_{1}(\mu')^{+} = [K_{1}(\mu')]^{+} = 1 = [K_{1}(\nu')]^{+}= J_{1}(\nu')^{+}.$

$2:$ If $j-k\notin K_{1}(\nu')\cap K_{1}(\mu')$ and $j-k\in K_{1}(\nu')\cup K_{1}(\mu')$, then either $K_{1}(\mu')^{-} < K_{1}(\nu')^{-}$ or  $K_{1}(\nu')^{-} < K_{1}(\mu')^{-}$. In either case, lemma \ref{lem:1} implies $\zeta(\mu')\neq \zeta(\nu')$, which is a contradiction.

$3:$ If $j-k\notin K_{1}(\nu')\cup K_{1}(\mu')$, then $K_{1}(\nu')^{-} > j-k$, $K_{1}(\mu') > j-k$ and $\nu'_{m} = \mu'_{m}$ for all $m\geq j-k$. Therefore, we have $[K_{1}(\nu')]^{-} = [K_{1}(\nu')]^{-}$. But this is a contradiction to the assumption in $3.$
\end{proof}

\begin{lem}
\label{lem:3}
Let $(G_B,E_A)$ be a regular KEP-action such that $B$ is in $M_N(\{0,1\})$. Suppose $\mu, \nu$ are in $E^{-\infty}_{A}$ and let $J \in \mathcal{I}(\mu)$ and $K \in \mathcal{I}(\nu)$ be the intervals containing $-1$, respectively. If $-\infty< J^-$, $J=K$, and $\Omega(\mu_{J}) = \Omega(\nu_{J})$, then $\theta^k(\mu_J) \neq \theta^k(\nu_J)$ implies $\zeta(\mu) \neq \zeta(\nu)$.
\end{lem}
\begin{proof}
Note that $\Omega(\mu_{J}) = \Omega(\nu_{K})$ implies $J= K$ is type $1$ for $\mu$ and $\nu$, or type $0$ for $\mu$ and $\nu$ . We denote this shared type as $k$. From $\Omega(\mu_{J}) = \Omega(\nu_{J})$, we have $A^{k}_{\mu_{J}} = A^{k}_{\nu_{J}} =: A$.  Therefore, we may write $\theta^k(\mu_{J}) =\frac{m}{A}$ and $\theta^k(\nu_{J}) = \frac{n}{A}$, for some $m, n\in \mathbb{N}\cup\{0\}$ such that $m,n < A$. By the hypothesis, we have $m\neq n$. Hence,
\[|e^{2\pi i \theta^{k}(\mu_{J})} - e^{2\pi i \theta^{k}(\nu_{J})}| = |1 - e^{\frac{2\pi i(n-m)}{A}}|\geq |1 - e^{\frac{2\pi i}{A}}|.\]
Denote $-J^{-} =: j $ and $M_k = M + 1-k$. We have $A\leq M_k^{j}$, and, since $\theta^k(\mu_J) \neq \theta^k(\nu_J)$, $A > 1$. Therefore, \[|1 - e^{\frac{2\pi i}{A}}|\geq |1 - e^{\frac{2\pi i}{M_k^{j}}}| = \sqrt{2 - 2\cos\left(\frac{2\pi}{M_k^{j}}\right)} = 2\sin\left(\frac{\pi}{M_k^{j}}\right).\]
Putting these two inequalities together, we have 
\begin{equation}
\label{sin}
|e^{2\pi i \theta^{k}(\mu_{J})} - e^{2\pi i \theta^{k}(\nu_{J})}|\geq 2\sin\left(\frac{\pi}{M_k^{j}}\right).
\end{equation}
Denote $\Omega(\mu_{J}) = \Omega(\nu_{J}) =:\omega$ and write $\zeta(\mu) = \frac{\omega}{R^{3}} e^{2\pi i \theta^{k}(\mu_{J})} + \frac{1}{R^{3j}}\zeta(\sigma^{j}(\mu))$, $\zeta(\nu) = \frac{\omega}{R^{3}} e^{2\pi i \theta^{k}(\nu_{J})} + \frac{1}{R^{3j}}\zeta(\sigma^{j}(\nu))$. Using \ref{sin},  and $|\zeta|\leq \frac{N}{R^{3}(R-1)}$,  we see that
\begin{equation*}\label{dmin-2rmax}
|\zeta(\mu) - \zeta(\nu)|\geq 2\frac{\omega}{R^{3}}\sin\left(\frac{\pi}{M_k^{j}}\right) - \frac{2N}{R^{3j + 3}(R-1)}.
\end{equation*}
From $\omega\geq \sum^{j+1}_{i=1}\frac{1}{R^{j}}\geq \frac{1}{2(R-1)}$ and $\sin(x)\geq x - \frac{x^{3}}{3!}$ for all $x\geq 0$, we have
\begin{equation*}\label{est1}
    2\frac{\omega}{R^{3}}\sin\left(\frac{\pi}{M_k^{j}}\right) - \frac{2N}{R^{3j + 3}(R-1)}\geq \frac{1}{(R-1)R^{3}}\left(\frac{\pi}{M_{k}^{j}} - \frac{\pi^{3}}{6 M_{k}^{3j}}\right) - \frac{2N}{R^{3j+3}(R-1)}.
\end{equation*} 
By multiplying the right hand side of inequality \eqref{est1} by $M^{3j}_{k}R^{3}(R-1)$ , we see that, by inequality \eqref{dmin-2rmax}, to prove the lemma, it suffices to prove
\begin{equation*}
    M_{k}^{2j}\pi - \frac{\pi^{3}}{6} > \frac{2NM^{3j}_{k}}{R^{3j}}.
\end{equation*}
Note that
\[
\frac{2NM^{3j}_{k}}{R^{3j}}\leq \frac{2N(M+1)^{3j}}{(N+1)^{3j}M^{3j}}\leq \frac{4(M+1)^{3j}}{2^{3j}M^{3j}}\leq 4.\]

Therefore,
\[M_{k}^{2j}\pi - \frac{\pi^{3}}{6} - \frac{2NM^{3j}_{k}}{R^{3j}}\geq M^{2j}_{k}\pi - \frac{\pi^{3}}{6} - 4 \geq 4\pi - \frac{\pi^{3}}{6} - 4 > 0. \qedhere\]
\end{proof}

\begin{thm}
\label{thm:embedding}
Let $(G_B,E_A)$ be a regular KEP-action such that $B\in M_N(\{0,1\})$. If $\mu, \nu \in E^{-\infty}_{A}$ are such that $\zeta(\mu) = \zeta(\nu)$, then either $\mu=\nu$ or there is $k<0$ such that $\mu_{[k,-1]} = \nu_{[k,-1]}$, $[-\infty, k-1]\in \mathcal{I}^{1}(\mu)\cap\mathcal{I}^{1}(\nu)$, $s(\mu_{j}) = s(\nu_{j})$ for all $j\leq k-1$ and $\theta^{1}(\mu_{[-\infty, k-1]}) = \theta^{1}(\nu_{[-\infty, k-1]})$.
\end{thm}

\begin{proof}
If $\mu=\nu$ we are done, so suppose $\mu \neq \nu$. Let $j<0$ be the largest number such that $\mu_j \neq \nu_j$ and let $J \in \mathcal{I}(\mu)$ and $K \in \mathcal{I}(\nu)$ be the intervals containing $j$ and note that Lemma \ref{lem:4} implies that $J^+=K^+$.

We claim that it suffices to show that $J^{-}= K^-= -\infty$ and  that $J=K$ is of type $1$. Indeed, If this is the case, then denoting $\mu' = \sigma^{|J^{+}+1|}(\mu)$ and $\nu' = \sigma^{|J^{+}+1|}(\nu)$, then $\frac{1}{R^{2}}\Omega(\mu_{J})e^{2\pi i \theta^1(\mu_{J})}=\zeta(\mu') =  \zeta(\nu') = \frac{1}{R^{2}}\Omega(\nu_{J})e^{2\pi i \theta^1(\nu_{J})}$. This implies that $\Omega(\mu_{J})=\Omega(\nu_{J})$ and $e^{2\pi i \theta^1(\mu_{J})}=e^{2\pi i \theta^1(\nu_{J})}$, which is the case if and only if $s(\mu_{j}) = s(\nu_{j})$ for all $j\leq k-1$ and $\theta^{1}(\mu_{[-\infty, k-1]}) = \theta^{1}(\nu_{[-\infty, k-1]})$.

Suppose either $-\infty< J$ or $-\infty < K$. Then, one of $J^{-} < K^{-}$ or $K^{-} < J^{-}$ or $J = K$ and $-\infty<J^{-}$. In the first two cases, Lemma \ref{lem:1} implies $\zeta(\mu')\neq\zeta(\nu')$, a contradiction. In the final case, Lemma \ref{lem:0} implies $\zeta(\mu')\neq \zeta(\nu')$ if $\Omega(\mu_{J})\neq \Omega(\nu_{J})$ and Lemma \ref{lem:3} implies $\zeta(\mu')\neq \zeta(\nu')$ if $\Omega(\mu_{J}) = \Omega(\nu_{J})$, a contradiction in both cases.

Therefore, we must have $J^{-} = K^{-} = -\infty$. For $\nu\in E_{A}^{-\infty}$, $\#(\nu_{i})\leq A_{\nu_{i}}^{0} -2$ for all $i < 0$. Hence, $e^{2\pi i\theta^{0}}:E^{-\infty}_{A}\to \mathbb{T}$ is injective. Therefore, if $J$ was type $0$, then $e^{2\pi i\theta^{0}(\mu_{J})} = e^{2\pi i\theta^{0}(\nu_{J})}$ implies $\mu_{J} = \nu_{J}$, a contradiction. So, $J$ is of type $1$, and the proof is complete.
\end{proof}

\begin{cor}[Proof of Theorem \ref{Main Embedding Thm}]
Let $(G_B,E_A)$ be a regular KEP-action such that $B\in M_N(\{0,1\})$. If $\mu, \nu \in E^{-\infty}_{A}$ are such that $\zeta(\mu) = \zeta(\nu)$, then $\mu \sim_{ae} \nu$. Therefore, $\zeta:\mathcal{J}_{G_B,E_A}\to \mathbb{C}$ is an embedding.
\end{cor}
\begin{proof}
    The characterisation of $\zeta(\mu) = \zeta(\nu)$ in Theorem \ref{thm:embedding} is equivalent to $\mu\sim_{ae}\nu$ (see Proposition \ref{ae_wrt_connected} or the proof of Proposition \ref{easy}).
\end{proof}

\end{document}